\documentclass[11pt]{amsart}
\usepackage{geometry}                
\usepackage{graphicx}
\usepackage{amssymb}
\usepackage{epstopdf}
\usepackage{amsmath}
\usepackage{color}
 \usepackage{bbm}
\usepackage{subfigure}

\usepackage[all]{xy}
\xyoption{matrix}
\xyoption{arrow}

\newtheorem{theorem}{Theorem}[subsection]
\newtheorem{lemma}[theorem]{Lemma}
\newtheorem{corollary}[theorem]{Corollary}
\newtheorem{proposition}[theorem]{Proposition}

\theoremstyle{definition}
\newtheorem{definition}[theorem]{Definition}
\newtheorem{example}[theorem]{Example}

\newtheorem{remark}[theorem]{Remark}
 
\numberwithin{equation}{section}

\newcommand{\mat}[1]{\ensuremath{
\left[\begin{matrix}#1
\end{matrix}\right]
}}

\def\xrarrow{\xrightarrow} 

\def\-{\text{-}}
 \newcommand{\into}{\hookrightarrow}
 \newcommand{\onto}{\twoheadrightarrow}
 \newcommand{\cof}{\rightarrowtail}

\def\<{\left<}
\def\>{\right>}

\newcommand{\Hom}{\ensuremath{\rm Hom}}%
\newcommand{\End}{\ensuremath{\rm End}}%
\newcommand{\ZZ}{{\ensuremath{\mathbb{Z}}}}

\newcommand{\RR}{{\ensuremath{\mathbb{R}}}}

\newcommand{\NN}{{{\mathbb{N}}}}

\newcommand{\commentout}[1]{}

\newcommand{\cA}{\ensuremath{{\mathcal{A}}}}

\newcommand{\cC}{\ensuremath{{\mathcal{C}}}}
\newcommand{\cD}{\ensuremath{{\mathcal{D}}}}

\newcommand{\cF}{\ensuremath{{\mathcal{F}}}}

\newcommand{\cL}{\ensuremath{{\mathcal{L}}}}
\newcommand{\cM}{\ensuremath{{\mathcal{M}}}}

\newcommand{\cP}{\ensuremath{{\mathcal{P}}}}

\newcommand{\cR}{\ensuremath{{\mathcal{R}}}}

\newcommand{\cT}{\ensuremath{{\mathcal{T}}}}

\def\a{\alpha}
\def\b{\beta}
\def\g{\gamma}

\def\e{\epsilon}
\def\f{\varphi}
\def\k{\kappa}

\def\Sig{\Sigma}

\def\th{\theta}

\title{Continuous Frobenius Categories}
\dedicatory{Dedicated to the memory of Dieter Happel} 
\author{Kiyoshi Igusa}
\address{Department of Mathematics, Brandeis University, Waltham, MA 02454}\email{igusa@brandeis.edu}

\author{Gordana Todorov}
\address{Department of Mathematics, Northeastern University, Boston, MA 02115}\email{g.todorov@neu.edu}

\thanks{The first author is supported by NSA Grant 111015.}

\begin{document}


\begin{abstract}{We introduce continuous Frobenius categories. These are topological categories which are constructed using representations of the circle over a discrete valuation ring. We show that they are Krull-Schmidt with one indecomposable object for each pair of (not necessarily distinct) points on the circle. By putting restrictions on these points we obtain various Frobenius subcategories. The main purpose of constructing these Frobenius categories is to give a precise and elementary description of the triangulated structure of their stable categories. We show in \cite{IT10} for which parameters these stable categories have cluster structure in the sense of \cite{BIRSc} and we call these \emph{continuous cluster categories}.}
\end{abstract}

\maketitle

\section*{Introduction}\label{sec:0}

The standard construction of a cluster category of a hereditary algebra is to take the orbit category of the derived category of bounded complexes of finitely generated modules over the algebra:
\[
	\cC_H\cong\cD^b(mod\,H)/F
\]
where $F$ is a triangulated autoequivalence of $\cD^b(mod\, H)$ \cite{BMRRT}. In this paper we construct continuous versions of the cluster categories of type $A_n$. These \emph{continuous cluster categories} are  \emph{continuously triangulated categories} (Sec \ref{sec0: TopCat}) having uncountably many indecomposable objects and containing the finite and countable cluster categories of type $A_n$ and $A_\infty$ as subquotients. Cluster categories of type $A_n$ and $A_\infty$ were also studied in \cite{CCS}, \cite{HJ12},  \cite{Ng}.

The reason for the term \emph{continuous} in the names of the categories is the fact that the categories that we define and consider in this paper are either topological categories, or equivalent to topological categories with continuous structure maps (Section \ref{sec0: TopCat}). The continuity requirement implies that there are two possible topologically inequivalent triangulations of the continuous cluster category given by the two 2-fold covering spaces of the Moebius band: connected and disconnected. We choose the first case (Definition \ref{defn: tilde Ind Ftop}) but we also discuss the second case (Remark \ref{another two fold cover}). These two topological categories are algebraically equivalent as triangulated categories by \cite{KellerReiten}.

The term \emph{cluster} in the names of the categories is justified in \cite{IT10} where it is shown that the category $\cC_\pi$ has a cluster structure where cluster mutation is given using the triangulated structure (see the beginning of Section \ref{sec3:CCC}) and that the categories $\cC_c$ also have a cluster structure for specific values of $c$. For the categories  $\cC_\phi$, we have partial results ($\cC_\phi$ has an $m$-cluster structure in certain cases). This paper is the first in a series of papers. The main purpose of this paper is to give a concrete and self-contained description of the triangulated structures of these continuous cluster categories being developed in concurrently written papers \cite{IT10,IT1x}.

We will use representations of the circle over a discrete valuation ring $R$ to construct  continuous Frobenius $R$-categories $\cF_\pi$, $\cF_c$ and $\cF_\phi$ whose stable categories (triangulated categories by a well-known construction of Happel \cite{HappelBook}) are equivalent to the continuous  categories $\cC_\pi$, $\cC_c$ and $\cC_\phi$, thus inducing continuous triangulated structure on these topological $K$-categories ($K=R/\mathfrak m$). 

In Section \ref{sec0: TopCat} we review the basic definition of a topological $R$-category and the topological additive category that it generates.

In Section \ref{RepCirc} we define representations of the circle; a representation of the circle $S^1=\RR/2\pi\ZZ$ over $R$ is defined  to be collection of  $R$-modules $V[x]$ at every point $x\in S^1$ and  morphisms $V[x]\to V[y]$ associated to any clockwise rotation from $x$ to $y$ with the property that rotation by $2\pi$ is multiplication by the uniformizer $t$ of the ring $R$. 
We denote the projective representations  generated at points $x$ by $P_{[x]}$. 

The Frobenius category $\cF_\pi$ is defined in Section \ref{sec: Frobenius categories}: the objects are $(V,d)$ where $V$ is a finitely generated projective representation of $S^1$ over $R$ and $d$ is an endomorphism of $V$ with square equal to multiplication by $t$.
We show that $\cF_\pi$ is a Frobenius category which has, up to isomorphism, one indecomposable object
\[
E(x,y)=\left( P_{[x]}\oplus P_{[y]}, \  d=\small\left[\begin{matrix}
0 & \b_\ast\\ \a_\ast & 0
\end{matrix}\right]\right)
\]
for every pair of points $0\le x\le y<2\pi$ in $S^1$. (See Definition \ref{defn of E(x,y)}.) The projective-injective objects are $E(x,x)$ (i.e. when $x=y$). The stable category of $\cF_\pi$ is shown to be equivalent to the continuous category $\cC_\pi$, which is defined in \ref{defn: continuous cluster category}. This construction also works in much greater generality (Proposition \ref{prop: general Frobenius category}).

We also consider $\cF_c$ for any positive real number $c\le\pi$ in \ref{cFc};
$\cF_c$ is defined to be the additive full subcategory of $\cF_\pi$ generated by all $E(x,y)$ where the distance from $x$ to $y$ is at least $\pi-c$. Objects  in $\cF_c$ are projective-injective if and only if they attain this minimum distance. The stable category is again triangulated and equivalent to the continuous category $\cC_c$ which has a cluster structure if and only if $c=(n+1)\pi/(n+3)$ for some $n\in\ZZ_{>0}$  \cite{IT10}. In that case we show (in \cite{IT10}) that $\cC_c$ contains a thick subcategory equivalent to the cluster category of type $A_n$. 

The most general version of Frobenius categories that we consider in this paper, are the categories $\cF_\phi$, for homeomorphisms $\phi: S^1\to S^1$ satisfying ``orientation preserving" and some other conditions (see \ref{def: Phi and F-Phi}). The categories $\cF_c$, and in particular $\cF_\pi$, are special cases of $\cF_\phi$.

In Section \ref{sec3:CCC} we define the topological Frobenius category $\cF_\pi^{top}$, $\cF_c^{top}$ and $\cF_\phi^{top}$ which are algebraically equivalent to $\cF_\pi, \cF_c,\cF_\phi$ and given by choosing two objects from every isomorphism class of indecomposable objects. The continuous cluster category $\cC_\pi$ is shown to be isomorphic to be a quotient category of $\cF_\pi^{top}$ and we give it the quotient topology.

In later papers we will develop other properties of these continuous cluster categories. We will give recognition principles for (the morphisms in) distinguished triangles in the continuous cluster categories and other continuous categories. An example of this is given in \ref{eg: distinguished triangle}. We will also show in later papers that the continuous cluster category $\cC_\pi$ has a unique cluster up to isomorphism and we will find conditions to make the cluster character into a continuous function. And we will show in later papers how this construction can be modified to produce continuous Frobenius categories of type $D$.

The first author would like to thank Maurice Auslander for explaining to him that ``The Krull-Schmidt Theorem is a statement about endomorphism rings [of objects in a category]''. This observation will be used many times. Also, Maurice told us that each paper should have only one main result. So, our other results will appear in separate papers. We also thank Adam-Christiaan van Roosmalen for explaining his work to us. Finally, the second author wishes to thank the Mathematical Sciences Research Institute (MSRI) in Berkeley, CA for providing support and wonderful working conditions during the final stages of writing this paper.

\setcounter{section}{-1}

\section{Some remarks on topological $R$-categories}\label{sec0: TopCat} We recall the definition of a topological category since our constructions are motivated by our desire to construct continuously triangulated topological categories of type $A$. 
By a ``continuously triangulated" category we mean a topological category which is also triangulated so that the defining equivalence $\Sig$ of the triangulated category is a continuous functor. We also review an easy method for defining the topology on an additive category out of the topology of a full subcategory of indecomposable objects.

Recall that a topological ring is a ring $R$ together with a topology on $R$ so that its structure maps are continuous. Thus addition $+:R\times R\to R$ and multiplication $\cdot:R\times R\to R$ are required to be continuous mappings. A \emph{topological $R$-module} is an $R$-module $M$ together with a topology on $M$ so that the structure maps $m:R\times M\to M$ and $a:M\times M\to M$ given by $m(r,x)= rx$ and $a(x,y)=x+y$ are continuous mappings.

\begin{definition} If $R$ is a topological ring, a \emph{topological $R$-category} is defined to be a small $R$-category $\cC$ together with a topology on the set of objects $Ob(\cC)$ and on the set of all morphisms $Mor(\cC)$ so that the structure maps of $\cC$ are continuous mappings. Thus $s,t,id,a,m,c$ are continuous where
\begin{enumerate}
\item $s,t:Mor(\cC)\to Ob(\cC)$ are the source and target maps.
\item $m:R\times Mor(\cC)\to Mor(\cC)$, $a:A\to Mor(\cC)$ are the mappings which give the $R$-module structure on each hom set $\cC(X,Y)\subset Mor(\cC)$. Here $A$ is the subset of $Mor(\cC)^2$ consisting of pairs $(f,g)$ of morphisms with the same source and target.
\item $id:Ob(\cC)\to Mor(\cC)$ is the mapping which sends each $X\in Ob(\cC)$ to $id_X\in \cC(X,X)\subseteq Mor(\cC)$.
\item $c:B\to Mor(\cC)$ is composition where $B$ is the subset of $Mor(\cC)^2$ consisting of pairs $(f,g)$ where $s(f)=t(g)$.
\end{enumerate}
\end{definition}

We say that a functor $F:\cC\to\cD$ between topological categories $\cC,\cD$ is \emph{continuous} if it is continuous on objects and morphisms. Thus, we require $Ob(F):Ob(\cC)\to Ob(\cD)$ and $Mor(F):Mor(\cC)\to Mor(\cD)$ to be continuous mappings. When $\cC,\cD$ are topological $R$-categories, we usually assume that $F$ is \emph{$R$-linear} in the sense that the induced mappings $\cC(X,Y)\to \cD(FX,FY)$ are homomorphisms of $R$-modules for all $X,Y\in Ob(\cC)$. 

In this paper we will construct Krull-Schmidt categories $\cC$ each of which has a natural topology on the full subcategory $\cD=Ind\,\cC$ of carefully chosen representatives of the indecomposable objects. By the following construction, we obtain a small topological category $add\,\cD$ which is equivalent as an additive category to the entire category $\cC$.

\begin{definition}\label{topology of add D} A topological $R$-category $\cD$ is called \emph{additive} if there is a continuous functor $\oplus:\cD\times\cD\to\cD$ which is algebraically a direct sum operation. ($\cD\times\cD$ is given the product topology on object and morphism sets.) 

Suppose $\cD$ is a topological $R$-category. Then we define the \emph{additive category} $add\,\cD$ generated by $\cD$ to be the category of formal ordered direct sums of objects in $\cD$. Thus the object space of $add\, \cD$ is the disjoint union
\[
	Ob(add\,\cD)=\bigsqcup_{n\ge0}Ob(\cD)^n.
\]
When $n=0$, $Ob(\cD)^0$ consists of a single object which we call the \emph{distinguished zero object} of $add\,\cD$. This is a topological space since it is the disjoint union of Cartesian products of topological spaces. We write the object $(X_i)$ as the ordered sum $\bigoplus_i X_i$. The morphism space is defined analogously:
\[
	Mor(add\,\cD)=\bigsqcup_{n,m\ge0} \{((X_j),(f_{ij}),(Y_i))\in Ob(\cD)^m\times Mor(\cD)^{nm}\times Ob(\cD)^n\ |\ f_{ij}\in \cD(X_j,Y_i)\}
\]
This has the topology of a disjoint union of subspaces of Cartesian products of topological spaces. 
\end{definition}

\begin{proposition}
The category $add\,\cD$ is a topological additive $R$-category in which direct sum $\oplus$ is strictly associative and has a strict unit.
\end{proposition}

\begin{proof}
Direct sum is strictly associative: $(A\oplus B)\oplus C=A\oplus (B\oplus C)$ since objects in $add\,\cD$ are, by definition, equal to ordered direct sums of objects in $\cD$. The distinguished zero object is a strict unit for $\oplus$ since it is the empty sum: $0 \oplus X=X=X\oplus 0$ for all $X$. The fact that $add\,\cD$ is a topological category follows easily from the assumption that $\cD$ is a topological category. For example, composition of morphisms in $add\,\cD$ is given by addition of composites of morphisms in $\cD$ and both of these operations are continuous.
\end{proof}

\section{Representations of the circle $S^1$}\label{RepCirc}
In this section we describe the category of representations of the circle over a discrete valuation ring. Special kinds of finitely generated projective representations  of the circle will be used in section 2 in order to define Frobenius categories.
Let $R$ be a discrete valuation ring with uniformizing parameter $t$ (a fixed generator of the unique maximal ideal $\mathfrak m$), valuation $\nu:R\to\NN$ and quotient field $K=R/\mathfrak m=R/(t)$.

\subsection{{Representations of $S^1$}}
Let $S^1=\RR/2\pi\ZZ$. Let  $x\in \RR$ and let $[x]$ denote the corresponding element $[x]=x+2\pi\ZZ$ in $S^1$. When we take an element $[x]\in S^1$ we mean: choose an element of $S^1$ and choose an arbitrary representative $x\in\RR$ of this element. We now define the category of $R$-representations of $S^1$. We denote this category $\cR_{S^1}$.

\begin{definition}
A \emph{representation} $V$ of $S^1$ over $R$ is defined to be:
\begin{enumerate}
\item[(a)] an $R$-module $V[x]$ for every $[x]\in S^1$ and
\item[(b)] an $R$-linear map $V^{(x,\a)}: V[x]\to V[{x-\a}]$ for all $[x]\in S^1$ and $\a\in \RR_{\ge0}$ 
 satisfying the following conditions for all $[x]\in S^1$:
\begin{enumerate}
\item [(1)] $V^{(x-\b,\a)}\circ V^{(x,\b)}=V^{(x,\a+\b)}$ for all $\a,\b\in\RR_{\ge0}$,
\item [(2)] $V^{(x,2\pi n)}:V[x]\to V[x]$ is multiplication by $t^n$ for all  $n\in\NN$.
\end{enumerate}\end{enumerate}
\end{definition}

\begin{definition} A \emph{morphism} $f:V\to W$ consists of  $R$-linear maps $f_{[x]}:V[x]\to W[x]$ for all $[x]\in S^1$ so that $W^{(x,\a)} f_{[x]}=f_{[x-\a]}V^{(x,\a)}$ for all $[x]\in S^1$ and $\a\ge0$, i.e.,
\[
\xymatrix{
V[x]\ar[d]^{V^{(x,\a)}}\ar[r]^{f_{[x]}} & W[x]\ar[d]^{W^{(x,\a)}}\\
V[x-\a]\ar[r]^{f_{[x-\a]}}&W[x-\a]	.
}
\]
A morphism $f:V\to W$ is called a \emph{monomorphism} or \emph{epimorphism} if $f_{[x]}:V[x]\to W[x]$ are  monomorphisms or epimorphisms, respectively, for all $[x]\in S^1$.
\end{definition}

\begin{definition}
Let $P_{[x]}$, for  $[x]\in S^1$, be the representation of $S^1$ defined as:
\begin{enumerate}
\item[(a)] $R$-module $P_{[x]}[x-\a]:=Re_x^\a$, the free $R$-module on one  generator $e_x^\a$ for each real number $0\le\a<2\pi$. We extend this notation to $\a\ge2\pi$ by $e_x^{\g+2\pi n}:=t^ne_x^\g\in P_{[x]}[x-\g]$ for $n\in\NN$, and $\g\in\RR_{\ge0}$.
\item[(b)] $R$-homomorphism $P_{[x]}^{(x-\a,\b)}:P_{[x]}[x-\a]\to P_{[x]}[x-\a-\b]$ is the unique $R$-linear map defined by $P_{[x]}^{(x-\a,\b)}(e_{x}^\a)=e_{x}^{\a+\b}$ for all $\b\in\RR_{\ge0}$.
\end{enumerate}
\end{definition}

\begin{remark} It follows from the definition  that $e_x^0$ is a generator of the representation  $P_{[x]}$;  we will often denote this generator by $e_x$.
\end{remark}
\begin{proposition}
Let $V$ be an $R$-representation of $S^1$. There is a natural isomorphism 
\[
\cR_{S^1}(P_{[x]},V)\cong V[x]
\]
given by sending $f:P_{[x]}\to V$ to $f_{[x]}(e_x)\in V[x]$. In particular the ring homomorphism $R\to\End(P_{[x]})\cong P_{[x]}[x]\cong R$ sending $r\in R$ to multiplication by $r$ is an isomorphism.
\end{proposition}

\begin{proof} Define a homomorphism $\f:  V[x] \to \cR_{S^1}(P_{[x]},V)$ in the following way. For every $v\in V[x]$ let $\f(v)\in \cR_{S^1}(P_{[x]},V)$ be given by $\f(v)_{[x-\a]}(re_x^\a):= V^{(x,\a)}(rv)\in V[x-\a]$ for all $0\le\a<2\pi$. Then $\f(v)$ is the unique morphism $P_{[x]}\to V$ such that $\f(v)(e_x)=v$. In particular, $\f(f_{[x]}(e_x))(e_x)=f_{[x]}(e_x))=f(e_x)$. Therefore $\f(f_{[x]}(e_x))=f$ since both morphisms send the generator $e_x\in P_{[x]}[x]$ to $f_{[x]}(e_x)$. Therefore, $\f$ gives an isomorphism $V[x]\cong \cR_{S^1}(P_{[x]},V)$ inverse to the map sending $f$ to $f_{[x]}(e_x)$.
\end{proof}

\begin{corollary}
Each representation $P_{[x]}$ is projective. In other words, if $f:V\to W$ is an epimorphism then $\Hom_R(P_{[x]},V)\to \Hom_R(P_{[x]},W)$ is surjective.\qed
\end{corollary}

If $x\le y<x+2\pi$ then $P_{[y]}[x]=R$ is generated by $e_y^\a$ where $\a=y-x$. So, we get the following Definition/Corollary.

\begin{definition}\label{DefDepth}
The \emph{depth} of any nonzero morphism of the form $f:P_{[x]}\to P_{[y]}$ is defined to be the unique nonnegative real number $\delta(f)=\a$ so that $f(e_x)=ue_y^\a$ for  a unit $u\in R$.  Since $t^ne_y^\a=e_y^{\a+2\pi n}$, this is equivalent to the formula $\delta(f)=\a+2\pi \nu(r)$ if $f(e_x)=re_y^\a$ and $\nu(r)\in\NN$ is the valuation of $r$. We define $\delta(0)=\infty$.
\end{definition}

\begin{lemma} The depth function $\delta$ has the following properties.
\begin{enumerate}
\item[(1)] For morphisms $f:P_{[x]}\to P_{[y]}$, $g:P_{[y]}\to P_{[z]}$ we have $\delta(g\circ f)=\delta(g)+\delta(f)$.
\item[(2)] For morphisms $f,g:P_{[x]}\to P_{[y]}$ and  $r,s\in R$ we have $\delta(rf+sg)\ge \min(\delta(f),\delta(g))$. 
\end{enumerate}
\end{lemma}

\begin{proof}
(1) If $f(e_x)=ue_y^\a$ and $g(e_y)=we_z^\b$ for units $u,w\in R$ then $gf(e_x)=uwe_z^{\a+\b}$ making $\delta(g\circ f)=\a+\b=\delta(f)+\delta(g)$.

(2) Let $f(e_x)=ue_y^\a$ and $g(e_x)=we_y^\b$ where $u,w$ are units in $R$. Suppose $\a=\delta(f)\le \b=\delta(g)$. Then $\b=\a+2\pi n$ for some $n\ge0$. So, $e_y^\b=t^ne_y^\a$ and $g(e_x)=wt^ne_y^\a$. So, $(rf+sg)(e_x)=(ru+swt^n)e_y^\a$ has depth $\ge\a=\min(\delta(f),\delta(g))$.
\end{proof}

We extend the definition of depth to any morphism $f:\bigoplus_i P_{[x_i]}\to \bigoplus_j P_{[y_j]}$ by
$$\delta(f):=min\{\delta(f_{ji})\ |\ f_{ji}:P_{[x_i]}\to P_{[y_j]} \}.$$

\begin{proposition}
The extended notion of depth satisfies the following conditions.
\begin{enumerate}
\item[(1)] Let $f:\bigoplus_i P_{[x_i]}\to \bigoplus_j P_{[y_j]}$, $g:\bigoplus_j P_{[y_j]}\to \bigoplus_k P_{[z_k]}$. Then $\delta(g\circ f)\ge\delta(g)+\delta(f)$.
\item[(2)] The depth of $f$ is independent of the choice of decompositions of the domain and range of $f$, i.e. $\delta(f)=\delta(\psi\circ f\circ \f)$ for all automorphisms $\psi,\f$ of $\bigoplus_j P_{[y_j]},\bigoplus_i P_{[x_i]}$.
\end{enumerate}
\end{proposition}

\begin{proof}
(1) By the extended definition of depth, $\delta(g f)$ is equal to the depth of one of its component functions $(gf)_{ki}:P_{[x_i]}\to P_{[z_k]}$. But this is the sum of composite functions of the form $g_{kj}f_{ji}:P_{[x_i]}\to P_{[y_j]}\to P_{[z_k]}$. By the lemma above, this gives
\[
\delta(g f)=\min(\delta((gf)_{ki}))\ge\min(\delta(g_{kj})+\delta(f_{ji}))\ge\delta(g)+\delta(f).
\]

(1) implies (2) since $\delta(\psi f\f)\ge \delta(\psi)+\delta(f)+\delta(\f)\ge\delta(f)$ and, similarly, $\delta(f)=\delta(\psi^{-1}\psi f\f\phi^{-1}g)\ge\delta(\psi f\f)$.
\end{proof}

\subsection{{Finitely generated projective representations of $S^1$}} It is shown here that finitely generated projective representations of $S^1$ are precisely the finitely generated torsion free representations.
\begin{definition}\label{def:depth}
A representation $V$ is \emph{torsion-free} if each $V[x]$ is a torsion-free $R$-module and each map $V^{(x,\a)}:V[x]\to V[x-\a]$ is a monomorphism. A representation $V$ is \emph{finitely generated} if it is a quotient of a finite sum of projective modules of the form $P_{[x]}$, i.e. there exists an epimorphism $\bigoplus_{i=0}^n P_{[x_i]}\onto V$.
\end{definition}

Let $V$ be a finitely generated torsion-free representation of $S^1$. Then the following lemma shows that a subrepresentation of $V$ generated at any finite set of points on the circle is a projective representation $P\cong \bigoplus m_iP_{[x_i]}$.

\begin{lemma}\label{lem: choice of vij is arbitrary} 
Let $V$ be a finitely generated torsion-free representation of $S^1$. Take any finite subset of $S^1$ and represent them with real numbers
\[
	x_0< x_1< x_2< \cdots < x_n<x_{n+1}=x_0+2\pi,\ \ \  x_i\in \RR.
\]
For each $0\le i\le n$ let $\{v_{ij}:{j=1,\cdots,m_i}\}$ be a subset of $V[x_i]$ which maps isomorphically to a basis of the cokernel of $V^{(x_{i+1},x_{i+1}-x_i)}:V[x_{i+1}]\to V[x_i]$ considered as a vector space over $K=R/(t)$. Let $f_{ij}:P_{[x_i]}\to V$ be the morphism 
defined by $f_{ij}(e_{x_i})=v_{ij}\in V[x_i]$
Then 
\begin{enumerate}
\item $f=\sum_{i=0}^n \sum_{j=0}^{m_i}f_{ij}:P=\bigoplus_{i=0}^n m_iP_{[x_i]}\to V$ is a monomorphism.
\item $f_{[x_i]}:P[x_i]\to V[x_i]$ is an isomorphism for each $i$.
\end{enumerate}
\end{lemma}
\begin{proof}
Since $V$ is torsion-free, the maps 
$V^{(x_i,x_i-x_0)}:V[x_i]\to V[x_0]$
are monomorphisms for  $i=0,1,2,\cdots,n$. Let $V_i=image(V^{(x_i,x_i-x_0)})\subset V[x_0].$
Then $$tV_0=V_{n+1}\subseteq V_n\subseteq\cdots \subseteq V_2\subseteq V_1\subseteq V_{0}.$$
 Furthermore, $V[x_i]\cong V_i$ and this isomorphism induces an isomorphism of quotients: \\$V[x_i]/V[x_{i+1}]\cong V_i/V_{i+1}$. 
 Let $w_{ij}\in V_i\subseteq V_0$ be the image of $v_{ij}\in V[x_i]$ and 
  let $\overline w_{ij}=w_{ij}+V_{i+1}\in V_i/V_{i+1}\cong V[x_i]/V[x_{i+1}]$.
   For each $i$, the $\overline w_{ij}$ form a basis for $V_i/V_{i+1}$. Taken together, $w_{ij}+tV_0$ form a basis for $V_0/tV_0$. Since $V_0$ is torsion free, it follows from Nakayama's Lemma, that the $w_{ij}$ generate $V_0$ freely. Therefore, the morphism $f:P=\bigoplus m_iP_{[x_i]}\to V$ which maps the generators of $P$ to the elements $v_{ij}$ induces an isomorphism $f_{[x_0]}:P[x_0]\cong V[x_0]$. 
 Applying the same argument to the points
 $$x_i< x_{i+1}< \cdots < x_n< x_0+2\pi, x_1+2\pi<\cdots < x_i+2\pi,  \ \ \  x_i\in \RR$$
   we see that $f_{[x_i]}:P[x_i]\to V[x_i]$ is an isomorphism for all $i$. This proves the second condition. The first condition follows.
\end{proof}
\begin{proposition}\label{prop:characterization of fg projective representations}
Every finitely generated projective representation of $S^1$ is torsion-free. Conversely, every  finitely generated torsion-free representation of $S^1$ over $R$ is projective and isomorphic to a direct sum of the form $\bigoplus_{i=0}^n P_{[x_i]}$.
\end{proposition}

\begin{proof}
 The first statement is clear since indecomposable projectives are torsion free and every direct sum of torsion-free representations is torsion-free. For the second statement, let $V$ be a finitely generated torsion-free representation of $S^1$. Suppose that $V$ is generated at $n+1$ points on the circle: $[x_0],[x_1],\cdots,[x_n]\in S^1$ where 
$x_0< x_1< x_2< \cdots < x_n<x_{n+1}=x_0+2\pi,$ with $x_i\in \RR$, having multiplicity $m_i$
as in the lemma. 
Let $f:P=\bigoplus_{i=0}^n m_iP_{[x_i]}\to V$ be the monomorphism given by the lemma. Then $f:P\to V$ is also onto by Condition (2) in the lemma since $V$ is generated at the points $[x_i]$. Therefore, $P\cong V$ as claimed.
\end{proof}

\subsection{{The category $\cP_{S^1}$}} Let $\cP_{S^1}$ be the category of all finitely generated  projective (and thus torsion-free) representations of $S^1$ over $R$. By the proposition above, each indecomposable object of $\cP_{S^1}$ is isomorphic to $P_{[x]}$ for some $[x]\in S^1$. 

\begin{lemma}\label{categorical epimorphism in PS1}
Any nonzero morphism $f:P_{[x]}\to P_{[y]}$ is a categorical epimorphism in $\cP_{S^1}$ in the sense that, for any two morphisms $g,h:P_{[y]}\to V$ in $\cP_{S^1}$, $gf=hf$ implies $g=h$.
\end{lemma}

\begin{proof}
Let $f(e_x)=re_y^{\a}$ for $r\neq0\in R$. Then $gf(e_x)=g(re_y^{\a})=g(rP_{[y]}^{(y,\a)}e_y)=rV^{(y,\a)}(g(e_y))$. By assumption this is equal to $hf(e_x)=rV^{(y,\a)}(h(e_y))$. Since $V$ is torsion-free, this implies that $g(e_y)=h(e_y)$ making $g=h$.
\end{proof}

For the proofs of Lemma \ref{lem:objects of smallest depth} and Proposition \ref{prop: projective and injective property of E(x,y) in ccc-c} below, we need the following easy observation using the depth $\delta(f)$ from Definition \ref{def:depth}.

\begin{proposition}\label{one more trivial observation}
Let $f:P_{[x]}\to P_{[y]}$.\begin{enumerate}
\item If $g:P_{[x]}\to P_{[z]}$ is a morphism so that $\delta(f)\le\delta(g)$ then there is a unique morphism $h:P_{[y]}\to P_{[z]}$ so that $hf=g$. 
\item If $g':P_{[w]}\to P_{[y]}$ is a morphism with $\delta(g')\ge\delta(f)$ then there is a unique $h':P_{[w]}\to P_{[x]}$ so that $fh'=g'$.\end{enumerate}
\end{proposition}

\begin{proof}
We prove the first statement. The second statement is similar. Let $\a=\delta(f),\b=\delta(g)-\a$. Then $f(e_x)=re_y^\a$ and $g(e_x)=se_z^{\a+\b}$ where $r,s$ are units in $R$. Let $h:P_{[y]}\to P_{[z]}$ be the morphism given by $h(e_y)=r^{-1}se_z^\b$. Then $hf(e_x)=rh(e_y^\a)=se_z^{\a+\b}$. So $hf=g$.
\end{proof}

\begin{definition}\label{defn: topology of PS1} We define $Ind\,\cP_{S^1}$ to be the topological $R$-category defined as follows. Algebraically, $Ind\,\cP_{S^1}$ is the full subcategory of $\cP_{S^1}$ with objects $P_{[x]}$ for every $[x]\in S^1$. We topologize this set of objects so that it is homeomorphic to $S^1$. The space of morphisms is the quotient space:
\[
	Mor(Ind\,\cP_{S^1})=\{(r,x,y)\in R\times \RR\times\RR\ |\ x\le y\le x+2\pi\}/\sim
\]
where the equivalence relation is given by $(r,x,y)\sim (r,x+2\pi n,y+2\pi n)$ for any $n\in\ZZ$ and $(r,x,x+2\pi)\sim (tr,x,x)$. Here $(r,x,y)$ represents the morphism $P_{[x]}\to P_{[y]}$ which sends $e_x$ to $re_y^{y-x}$. The second relation comes from the identity $re_x^{2\pi}=tre^0_x$. We give $\RR$ the usual topology and $R$ the $\mathfrak m$-adic topology.
\end{definition}

\begin{remark} \begin{enumerate}\item The category $\cP_{S^1}$ is algebraically equivalent to the topological additive $R$-category $add\,Ind\,\cP_{S^1}$ given by Definition \ref{topology of add D}.
\item In the terminology of \cite{vR}, $\cP_{S^1}$ is the full subcategory of finitely generated projective objects in the \emph{big loop} $\widehat{K\cL^\bullet}$ where $\cL$ is the half-open interval $[0,2\pi)$ considered as a linearly ordered set.
\end{enumerate}
\end{remark}

\section{The Frobenius categories $\cF_\pi$, $\cF_c$, $\cF_\phi$}\label{sec: Frobenius categories}

We define first $\cF_\pi$ as the most natural Frobenius category coming from the representations of the circle $S^1=\RR/2\pi\ZZ$. The categories $\cF_c,\cF_\phi$ will be defined as certain full subcategories of $\cF_\pi$.

\subsection{{Frobenius category $\cF_\pi$}} We define the category $\cF_\pi$ and the set of exact sequences in $\cF_\pi$. Then we show that $\cF_\pi$ is an exact category and that it has enough projectives with respect to the exact structure. Finally, we show that projective and injective objects in $\cF_\pi$ coincide proving that $\cF_\pi$ is a Frobenius category.

\begin{definition}
The category  $\cF_\pi$ and the exact sequences in $\cF_\pi$ are defined as:
\begin{enumerate}
\item \emph{Objects} of  $\cF_\pi$ are pairs $(V,d)$ where $V\in \cP_{S^1}$ and 
 $d:V\to V$ is an endomorphism of $V$ so that $d^2=t$ (multiplication by $t$).

\item \emph{Morphisms} in $\cF_\pi$ are  $f:(V,d)\to (W,d)$ where $f:V\to W$  satisfies  $fd=df$. 
\item  \emph{Exact sequences} in $\cF_\pi$ are
$(X,d)\xrarrow f(Y,d)\xrarrow g(Z,d)$
where  $0\to X\xrarrow f Y\xrarrow g Z\to0$ is exact (and therefore split exact) in $\cP_{S^1}$.
\end{enumerate}
Following Waldhausen \cite{W} we call the first morphism in an exact sequence a \emph{cofibration} and write it as $(X,d)\cof (Y,d)$ and we call the second morphism a \emph{quotient map} and denote it by $(Y,d)\onto(Z,d)$.
\end{definition}

\begin{remark} Note that if $(V,d)$ is an object in $\cF_\pi$, then $V$ cannot be indecomposable since $\End(P_x)=R$ does not contain an element whose square is $t$. We will see later that $V$ must have an even number of components. 
\end{remark}

\begin{theorem}\label{thmxxx}
 The category $\cF_\pi$ is a Frobenius category.
\end{theorem}

The proof of this theorem will occupy the rest of this subsection.

\begin{lemma}
A morphism $f:(V,d)\cof (W,d)$ in $\cF_\pi$ is a cofibration if and only if $f:V\to W$ is a split monomorphism in $\cP_{S^1}$. Similarly, $f$ is a quotient map in $\cF_\pi$ if and only if it is a split epimorphism in $\cP_{S^1}$. In particular, all epimorphisms in $\cF_\pi$ are quotient maps.
\end{lemma}

\begin{proof}
By definition of exactness, the split monomorphism condition is necessary. Conversely, suppose that $f:V\to W$ is split mono in $\cP_{S^1}$. Then the cokernel $C$ is projective, being a summand of the projective object $W$. Since $fd=df$, we have an induced map $d:C\to C$. Since $d^2=t$ on $V$ and $W$ we must have $d^2=t$ on $C$. Therefore, $f$ is the beginning of the exact sequence $(V,d)\cof (W,d)\onto (C,d)$. The other case is similar with the added comment that all epimorphisms in $\cP_{S^1}$ are split epimorphisms.
\end{proof}

\begin{lemma}
$\cF_\pi$ is an exact category.
\end{lemma}

\begin{proof}
We verify the dual of the short list of axioms given by Keller \cite{K}. The first two axioms follow immediately from the lemma above.

(E0) $0\cof 0$ is a cofibration.

(E1) The collection of cofibrations is closed under composition.

(E2) The pushout of an exact sequence $\xymatrixrowsep{10pt}\xymatrixcolsep{12pt}
\xymatrix{
(A,d)\ \ar@{>->}[r]^f &
	(B,d)\ar@{->>}[r]^g &
	(C,d)
	}
$ along any morphism $h:(A,d)\to (A',d)$ exists and gives an exact sequence $(A',d)\cof (B',d)\onto (C,d)$.

\emph{Pf:} Since $f:A\to B$ is a split monomorphism in $\cP_{S^1}$, so is $(f, h):A\to B\oplus A'$. By the Lemma, we can let $(B',d)\in\cF_\pi$ be the cokernel of $(f,h)$. Since the pushout of a split sequence is split, the sequence $A'\to B'\to C$ splits in $\cP_{S^1}$. Therefore $(A',d)\cof (B',d)\onto (C,d)$ is an exact sequence in $\cF_\pi$. Similarly, we have the dual axiom:

(E2)$^{op}$ The pullback of an exact sequence in $\cF_\pi$ exists and is exact.

Therefore, $\cF_\pi$ is an exact category.
\end{proof}

We record the following easy extension of this lemma for future reference.

\begin{proposition}\label{prop: exact subcategories of ccc}
Suppose that $\cA$ is an additive full subcategory of $\cF_\pi$ with the property that any cofibration in $\cF_\pi$ with both objects in $\cA$ has cokernel in $\cA$ and that any quotient map in $\cF_\pi$ with both objects in $\cA$ has kernel in $\cA$. Then $\cA$ is an exact subcategory of $\cF_\pi$.
\end{proposition}

\begin{proof}
Under the first condition, cofibrations in $\cA$ will be closed under composition and under pushouts since the middle term of the pushout of $X\cof Y\onto Z$ under any morphism $X\to X'$ in $\cA$ is the cokernel of the cofibration $X\cof Y\oplus X'$. Dually, quotient maps will be closed under pull-backs since the pull-back of a quotient map $Y\onto Z$ along a morphism $W\to Z$ is the kernel of the quotient map $Y\oplus W\onto Z$. So, $\cA$ is exact.
\end{proof}

\begin{definition}\label{defn of E(x,y)} Let $P$ be an object of $\cP_{S^1}$. We define the object $P^2\in\cF_\pi$ to be
\[
	P^2:=\left(
	P\oplus P,\small\left[\begin{matrix}
	0 & t\\1 &0
	\end{matrix}\right]
	\right).
\]
\end{definition}

It is clear that $(P\oplus Q)^2=P^2\oplus Q^2$, hence the functor $(\,)^2:\cP_{S^1}\to\cF_\pi$ is additive. We will show that this functor is both left and right adjoint to the forgetful functor $\cF_\pi\to\cP_{S^1}$ which sends $(V,d)$ to $ V$.

\begin{lemma}
${\cF_\pi}(P^2,(V,d))\cong {\cP_{S^1}}(P,V)$ and $P^2$ is projective in $\cF_\pi$.
\end{lemma}

\begin{proof}
A morphism $P^2\to (V,d)$ is the same as a pair of morphisms $f,g:P\to V$ so that $g=df$. So, $(f,df)\leftrightarrow f$ gives the desired isomorphism. To see that $P^2$ is projective in $\cF_\pi$, consider any quotient map $(V,d)\onto (W,d)$ and a morphism $(f,df):P^2\to (W,d)$. We can choose a lifting $\tilde f:P\to V$ of $f:P\to W$ to get a lifting $(\tilde f,d\tilde f)$ of $(f,df)$.
\end{proof}

\begin{lemma}
${\cF_\pi}((V,d),P^2)\cong {\cP_{S^1}}(V,P)$ and $P^2$ is injective for cofibrations in $\cF_\pi$.
\end{lemma}

\begin{proof} 
A morphism $(V,d)\to P^2$ is the same as a pair of morphisms $f,g: V\to P$ so that $f=gd$. Therefore, $(gd,g)\leftrightarrow g$ gives the isomorphism. 
To see that $P^2$ is injective for cofibrations, consider any cofibration $(V,d)\cof (W,d)$ and any morphism $(gd,g):(V,d)\to P^2$. Then, an extension of $(gd,g)$ to $(W,d)$ is given by $(\overline g d,\overline g)$ where $\overline g:W\to P$ is an extension of $g:V\to P$ given by the assumption that $V\to W$ is a split monomorphism.
\end{proof}

\begin{lemma}
The category $\cF_\pi$ has enough projective and injective objects: $V^2$, $V\in \cP_{S^1}$.
\end{lemma}

\begin{proof}
For any object $(V,d)\in\cF_\pi$ the projective-injective object $V^2$ maps onto $(V,d)$ by the quotient map $(1,d):V^2\onto (V,d)$. Also  $(d,1):(V,d)\cof V^2$ is a cofibration.
\end{proof}

\begin{proof}[of Theorem \ref{thmxxx}]
We only need to show that every projective object in $\cF_\pi$ is isomorphic to an object of the form $P^2$ for some $P\in\cP_{S^1}$ and is therefore injective.

Let $(V,d)$ be a projective object in $\cF_\pi$. Then the epimorphism $(1,d):V^2\to (V,d)$ splits. Therefore, $(V,d)$ is isomorphic to a direct summand of $V^2$. By Proposition \ref{prop:characterization of fg projective representations}, the representation $V$ decomposes as $V\cong \bigoplus_{i=0}^n P_{[x_i]}$. It follows that $V^2\cong \bigoplus_{i=0}^n P_{[x_i]}^2$. Therefore, $(V,d)$ is a direct summand of $\bigoplus_{i=0}^n P_{[x_i]}^2$. We need a Krull-Schmidt theorem to let us conclude that $(V,d)$ is isomorphic to a direct sum of a subset of the projective objects $P_{[x_i]}^2$. This follows from the following lemma.
\end{proof}

\begin{lemma}\label{End(P2) is local}
The endomorphism ring of $P_{[x]}^2$ is a commutative local ring. Therefore, every indecomposable summand of $\bigoplus P_{[x_i]}^2$ is isomorphic to one of the terms $P_{[x_i]}^2$.
\end{lemma}

\begin{proof}
By the two previous lemmas, an endomorphism of $P_{[x]}^2$ is given by morphism
\[
	\left[\begin{matrix}
	a & tb\\ b&a
	\end{matrix}\right]:P_{[x]}\oplus P_{[x]}\to P_{[x]}\oplus P_{[x]}
\]
where $a,b\in \End(P_{[x]})=R$. Calculation shows that matrices of this form commute with each other. Those matrices with $a\in (t)$ form an ideal and, if $a\notin (t)$ then
\[
	\left[\begin{matrix}
	a & tb\\ b&a
	\end{matrix}\right]^{-1}=\left[\begin{matrix}
	au & -tbu\\ -bu&au
	\end{matrix}\right]
\]
where $u$ is the inverse of $a^2-tb^2$ in $R$. Therefore, $\End_{\cF_\pi}(P_{[x]}^2)$ is local.
\end{proof}

\subsection{{Indecomposable objects in $\cF_\pi$}} We now describe  representations $E(x,y)$ and prove that all indecomposable objects of $\cF_\pi$ are isomorphic to these representations. 

\begin{definition}
Let $[x], [y]$ be two (not necessarily distinct) elements of $S^1$ and represent them by real numbers $x\le y\le x+2\pi$. Let $\a=y-x, \b=x+2\pi-y$ and let
\[
E(x,y)=\left( P_{[x]}\oplus P_{[y]}, \  d=\small\left[\begin{matrix}
0 & \b_\ast\\ \a_\ast & 0
\end{matrix}\right]\right)
\]
where $\a_\ast:P_{[x]}\to P_{[y]}$ is the morphism 
such that $\a_\ast (e_x)=e_y^\a$
for the generator $e_x\in P_{[x]}[x]$ and $e_y^\a\in P_{[y]}[x]$ and, similarly, $\b_\ast:P_{[y]}\to P_{[x]}$ sends $e_y\in P_{[y]}[y]$ to $e_x^\b\in P_{[x]}[y]$. In other words, $d(re_x^\g,se_y^\delta)=(se_x^{\delta+\b},re_y^{\g+\a})$ for all $r,s\in R$ and $\g,\delta\ge0$. 
\end{definition}

There is an isomorphism $E(x,y)\cong E(y,x+2\pi)$ given by switching the two summands and an equality $E(x,y)=E(x+2\pi n,y+2\pi n)$ for every integer $n$. In the special case $x=y$, we have $\a=0$ making $\a_\ast$ the identity map on $P_{[x]}$ and $\b=2\pi$ making $\b_\ast$ equal to multiplication by $t$. Thus, $E(x,x)=P_{[x]}^2$ and $E(x,x+2\pi)\cong P_{[x]}^2$ which is projective in $\cF_\pi$.

\begin{lemma}
The endomorphism ring of $E(x,y)$ is a commutative local ring. Therefore, $E(x,y)$ is an indecomposable object of $\cF_\pi$.
\end{lemma}
\begin{proof}
Computation shows that endomorphisms of $E(x,y)$ are given by matrices $\small\left[\begin{matrix} a&tb\\b&a\end{matrix}\right]$ with $a,b\in R$. 
Therefore End$_{\cF_\pi}(E(x,y))$ is a commutative local ring as in  Lemma \ref{End(P2) is local}.
\end{proof}

In order to prove that the category $\cF_\pi$ is Krull-Schmidt we need the following lemma, which uses the notion of depth as defined in \ref{DefDepth}

\begin{lemma}\label{lem:objects of smallest depth}
Let $(V,d)$ be an object  in $\cF_\pi$ and let $\f:V\cong\bigoplus_{i=0}^n  P_{[x_i]}$ be a decomposition of $V$ into indecomposable summands. 
 Let $f_{ji}:P_{[x_i]}\to P_{[x_j]}$ be one of the components of $f=\f d\f^{-1}$ with the smallest depth. Then we may choose $i\neq j$, and the representatives $x_i, x_j\in \RR$ so that  $x_i\le x_j\le x_i+\pi$ and $E(x_i,x_j)$ is a direct summand of $(V,d)$.
\end{lemma}

\begin{proof}
We first note that, since the depth of $d^2=t$ is $2\pi$, the depth of $d$ is $\delta(d)\le \pi$. Therefore $\delta(f)\le\pi$. Next, we show that the minimal depth is attained by an off-diagonal entry of the matrix $(f_{ji}:P_{[x_i]}\to P_{[x_j]})$. Suppose that a diagonal entry $f_{ii}$ has the minimal depth. Then $\delta(f_{ii})=0$ (since it can't be $2\pi$). But then $f_{ii}$ is an isomorphism. But $f^2$ is zero modulo $t$. To cancel the $f_{ii}^2$ term in $f^2$ there must be some $j\neq i$ so that $f_{ji}$ is also an isomorphism, making $\delta(f_{ji})=0$. 

So, we may assume that $P_{[x_i]}$ and $P_{[x_j]}$ are distinct components of $V$ and we may choose the representatives $x_i,x_j$ in $\RR$ so that $x_i\le x_j\le x_i+\pi$ and 
$\delta(d)=\delta(f)=\delta(f_{ji})=x_j-x_i$. 
Let $\a=x_j-x_i$ and $\b=x_i+2\pi-x_j=2\pi-\a$.
We now construct a map $\rho: E(x_i,x_j)\to V$, $$\left(P_{[x_i]}\oplus P_{[x_j]}, d_E=\small\left[\begin{matrix}
0 & \b_\ast\\ \a_\ast & 0
\end{matrix}\right]\right)\xrightarrow{\rho}(V,d); $$
 $\a_\ast: P_{[x_i]}\to P_{[x_j]}$ is defined by $\a_\ast(e_{x_i})=e_{x_j}^{\a}$ and  
 $\b_\ast: P_{[x_j]}\to P_{[x_i]}]$  by $\b_\ast(e_{x_j})=e_{x_i}^{\b}$. 
So $\delta(\a_\ast)=\a$ and $\delta(\b_\ast)=\b$. In order to define $\rho$
consider the following diagram where top squares commute by the definition of $f$. The existence and uniqueness of the map $h:P_{[x_j]}\to \bigoplus_k P_{[x_k]}$ follows by Proposition \ref{one more trivial observation} since $\delta(\a_\ast)=\a\le\delta(f\circ incl_i)$.
\[
\xymatrixrowsep{18pt}\xymatrixcolsep{20pt}
\xymatrix{
V \ar[r]^{d}\ar[d]_{\f}^{\cong}& 
	V\ar[r]^d \ar[d]_{\f}^{\cong}&
	V\ar[d]_{\f}^{\cong}\\
\bigoplus_k P_{[x_k]} \ar[r]^{f}& 
	\bigoplus_k P_{[x_k]}\ar[r]^f&
	\bigoplus_k P_{[x_k]}\\
P_{[x_i]}\ar[u]^{incl_i} \ar[r]^{\a_\ast }& 
	 P_{[x_j]}\ar[r]^{\b_\ast}\ar[u]^{\exists !h}&
	P_{[x_i]}\ar[u]^{incl_i}
		}
\]
Notice that (1) $f^2 incl_i = incl_i \b_\ast \a_\ast$ since both maps are multiplications by $t$; reasons:\\
$f^2=\f d^2 \f^{-1}$ and therefore is multiplication by $t$,
and since $\delta(\b_\ast \a_\ast)=2\pi$ the map $\b_\ast\a_\ast$ is also multiplication by $t$. From $f\, incl_i= h\a_\ast$  and (1) it follows that $fh\a_\ast = incl_i \b_\ast \a_\ast$. Since $\a_\ast$ is a categorical epimorphism in $\cP_{S^1}$, it follows that the bottom right square commutes, i.e. 
$fh= incl_i \b_\ast$. 

Define $\rho:=(\f^{-1} incl_i, \f^{-1}h)$ and check that $d\circ\rho=\rho\circ d_E$.
Then $d\circ\rho =(d\f^{-1} incl_i, d\f^{-1}h)$, and $\rho\circ d_E=(\f^{-1} incl_i, \f^{-1}h)\circ d_E= (\f^{-1}h\circ \a_\ast, \f^{-1} incl_i\circ \b_\ast)=(d\f^{-1} incl_i, d\f^{-1}h)=d\circ\rho$.
Similarly we get the diagram
\[
\xymatrixrowsep{18pt}\xymatrixcolsep{20pt}
\xymatrix{
V \ar[r]^{d}\ar[d]_{\f}^{\cong}& 
	V\ar[r]^d \ar[d]_{\f}^{\cong}&
	V\ar[d]_{\f}^{\cong}\\
\bigoplus_k P_{[x_k]} \ar[d]_{proj_j} \ar[r]^{f}& 
	\bigoplus_k P_{[x_k]}\ar[r]^f\ar[d]_{\exists !g}&
	\bigoplus_k P_{[x_k]}\ar[d]_{proj_j}\\
P_{[x_j]} \ar[r]^{\b_\ast }& 
	 P_{[x_i]}\ar[r]^{\a_\ast}&
	P_{[x_j]}
		}
\]
and the map $\rho':(V,d)\to (P_{[x_i]}\oplus P_{[x_j]},d_E)$ defined as 
$\rho'=(g\f,proj_j\f)$.Then $\rho'\circ d=d_E\circ \rho'$.
Then the composition $E(x_i,x_j)\xrightarrow{\rho}(V,d)\xrightarrow{\rho'}E(x_i,x_j)$ is 
$$\rho'\rho=\left[\begin{matrix}
g\circ incl_i & g\circ h\\ proj_j\circ incl_i & proj_j\circ h
\end{matrix}\right]$$
 which is an isomorphism since both $i$-th component of $g$ and $j$-th components of $h$ are isomorphisms making the diagonal entries of this matrix invertible as in the proof of Lemma \ref{End(P2) is local}. So $E(x_i,x_j)$ is isomorphic to a summand of $(V,d)$.
\end{proof}

\begin{theorem}\label{thm: indec objects of ccc}
The category $\cF_\pi$ is a Krull-Schmidt category with indecomposable objects isomorphic to $E(x,y)$ for some $0\le x\le y<2\pi$.
\end{theorem}
\begin{proof}
The Lemma \ref{lem:objects of smallest depth} implies the theorem since it shows, by induction on the number of components of $V$, that $(V,d)$ is a direct sum of indecomposable objects $E(x,y)$. 
\end{proof}
\begin{corollary}Indecomposable projective-injective objects in $\cF_\pi$ are isomorphic to $E(x,x)$ for some $x$.
\end{corollary}


\subsection{{Support intervals}} We will formulate an extension of Lemma \ref{lem:objects of smallest depth} which will be useful for constructing other Frobenius categories. To do this we replace depth conditions with conditions on the ``support intervals'' of a morphism.

A \emph{closed interval} in $S^1$ is defined to be a closed subset of the form $[x,y]$ where $x\le y<x+2\pi$. These subsets are characterized by the property that they are nonempty, compact and simply connected. For example, a single point is a closed interval.

Let $P=\bigoplus P_{[x_i]}$, $Q=\bigoplus P_{[y_j]}$ be objects of $\cP_{S^1}$ with given decompositions into indecomposable objects. Let $f:P\to Q$ be a map with components $f_{ji}:P_{[x_i]}\to P_{[y_j]}$. Then $\delta(f_{ji})=y_j-x_i+2\pi n$ for some $n\in\NN$ (or $\delta(f_{ji})=\infty$). Consider the collection of all closed intervals $[x_i,y_j]\subsetneq S^1$ with the property that $\delta(f_{ji})=y_j-x_i$. A minimal element of this collection (ordered by inclusion) will be called a \emph{support interval} for $f$. The collection of all support intervals is the \emph{support} of $f$.

As an example, if $E(x,y)=(P,d)$ where $x\neq y$ then the support of $d$ consists of the intervals $[x,y]$ and $[y,x+2\pi]$.

\begin{proposition} Let $f: P\to Q$ be a morphism in $\cP_{S^1}$. 
The support of $f$ is independent of the choice of decompositions of $P$ and $Q$.
\end{proposition}

\begin{proof}
Let $f'=\psi\circ f\circ \f$ where $\f,\psi$ are automorphisms of $P,Q$ respectively. Let $[x_i,y_j]$ be a support interval for $f'$. Then we have a nonzero composition:
\[
	P_{[x_i]}\to \bigoplus P_{[x_a]}\xrarrow f \bigoplus P_{[y_b]}\to P_{[y_j]}
\]
of depth $<2\pi$. This implies that $f$ has a support interval $[x_a,y_b]$ for some $a,b$ and  $[x_a,y_b]\subset[x_i,y_j]$. Therefore, every support interval of $f'$ contains a support interval of $f$. The reverse is also true by symmetry. So, the supports of $f,f'$ are equal.
\end{proof}

\begin{remark}
Note that the depth $\delta(f)$ of a morphism $f:P\to Q$ in $\cP_{S^1}$ is equal to  the minimum length $y-x$ for all support intervals $[x,y]$ of $f$ when $f$ has nonempty support and $\delta(f)\ge2\pi$ otherwise.
\end{remark}

The following proposition generalizes Lemma \ref{lem:objects of smallest depth} above.

\begin{proposition}\label{prop: projective and injective property of E(x,y) in ccc-c}
Let $(V,d)\in \cF_\pi$. Suppose that $x\le y< x+2\pi$ and the closed interval $[x,y]\subset S^1$ does not properly contain any support interval of $d:V\to V$. Then
\begin{enumerate}
\item ${\cF_\pi}(E(x,y),(V,d))\cong \cP_{S^1}(P_{[x]},V)$, the isomorphism is given by restriction to the component $P_{[x]}$ of $E(x,y)$.
\item ${\cF_\pi}((V,d),E(x,y))\cong \cP_{S^1}(V,P_{[y]})$, the isomorphism is given by projection to $P_{[y]}$. 
\end{enumerate}
\end{proposition}

\begin{proof}
These statements follow from Proposition \ref{one more trivial observation}, as illustrated in the two diagrams in the proof of Lemma \ref{lem:objects of smallest depth} above.
\end{proof}


\subsection{{The Frobenius categories $\cF_c,\cF_\phi$}}\label{cFc} Let $c, \th\in\RR_{>0}$ be such that 
$c+\th=\pi$ and let $\cF_c$ denote the full subcategory of $\cF_\pi$ whose objects are all $(V,d)$ with the property that the depth of $d$ is $\delta(d)\ge \th$. 
We  show that $\cF_c$ is a Frobenius category whose stable category is equivalent to the category $\cC_c$ defined in next section and discussed in detail in later papers in this series. 
In particular, the category $\cC_c$ will be shown to be a cluster category (without coefficients or frozen objects) if and only if $\th=2\pi/(n+3)$ for $n\in \ZZ_{>0}$. This is equivalent to $c=(n+1)\pi/(n+3)$. The category $\cF_c$ is a special case of the following more general construction which produces many examples of cluster and $m$-cluster categories as we will show in other papers.

\begin{definition}\label{def: Phi and F-Phi}
The category $\cF_\phi$ is defined as the full subcategory of $\cF_\pi$ consisting of all $(V,d)$ so that every support interval $[x,y]$ of $d$ contains an interval of the form $[z,\phi(z)]$, where $\phi:\RR\to \RR$ is a homeomorphism of the real line to itself satisfying:\begin{enumerate}
\item $\phi(x+2\pi)=\phi(x)+2\pi$. 
\item $x\le \phi(x)<x+\pi$ for all $x\in \RR$.
\end{enumerate}
\end{definition}

The first condition implies that $\phi$ induces an orientation preserving homeomorphism $\overline\phi$ of the circle $S^1$ to itself. The second condition says that $\overline\phi$ ``does not move points clockwise'' and also implies that $\phi^2(x)<x+2\pi$. The condition on the support interval $[x,y]$ is equivalent to the condition $\phi(x)\le y$. In the special case when $\phi(x)=x+\th$ where $\th=\pi-c$, this condition is: $y\ge x+\th$. So, $\cF_\phi=\cF_c$ in this case.

\begin{proposition}\label{prop: indecomposable objects of F-Phi} The category
$\cF_\phi$ is a Krull-Schmidt category with indecomposable objects isomorphic to $E(x,y)$ satisfying $\phi(x)\le y$ and $\phi(y)\le x+2\pi$.\qed
\end{proposition}

To prove that $\cF_\phi$ in general, and $\cF_c$ in particular, is a Frobenius category, the following observation is helpful.

\begin{lemma}\label{lem: Px-Px isomorphism implies x<z<y}
\begin{enumerate}
\item Suppose that $x<y<z<x+2\pi$. Then a morphism $f:P_{[x]}\to P_{[y]}$ factors through $P_{[z]}$ if and only if $f=tg$ for some $g:P_{[x]}\to P_{[y]}$.
\item If $f:E(x,z)\to E(x,y)$ is a morphism whose $P_{[x]}-P_{[x]}$ component is an isomorphism then $x\le z\le y$. 
\end{enumerate}
\end{lemma}

\begin{proof}
If $f:P_{[x]}\to P_{[y]}$ factors through $P_{[z]}$ then its depth must be at least $y-x+2\pi$. So, it is divisible by $t$. Conversely, any morphism which is divisible by $t$ factors through $P_{[w]}$ for all points $[w]\in S^1$. This proves (1) and (1) implies (2).
\end{proof}

\begin{theorem}\label{thm:FJ is Frobenius}
The category $\cF_\phi$ is a Frobenius category with projective-injective objects $E(x,y)$ where either $y=\phi(x)$ or $x+2\pi=\phi(y)$.
\end{theorem}

\begin{proof}
To show that $\cF_\phi$ is an exact category it suffices, by Proposition \ref{prop: exact subcategories of ccc}, to show that a cofibration in $\cF_\pi$ with both objects in $\cF_\phi$ has cokernel in $\cF_\phi$ and similarly for kernels. So, let $\xymatrixrowsep{10pt}\xymatrixcolsep{12pt}
\xymatrix{
(X,d)\ \ar@{>->}[r]^f &
	(Y,d)\ar@{->>}[r]^g &
	(Z,d)
	}
$
be an exact sequence in $\cF_\pi$ so that $(X,d),(Y,d)$ lie in $\cF_\phi$ and let $E(x,y)$ be a component of $(Z,d)$. If $E(x,y)$  is  a component of $(Y,d)$, then it is in $\cF_\phi$.
Now suppose $E(x,y)$ is a component of $(Z,d)$ but not of $(Y,d)$. Since $Y\to Z$ is split epimorphism in $\cP_{S^1}$, there are components $E(x,a),E(y,b)$ of $(Y,d)$ so that: 
\begin{enumerate}
\item[(a)] $P_{[x]}\subseteq E(x,a)$ and  $g(P_{[x]})\cong P_{[x]}\subseteq E(x,y)$ and
\item[(b)] $P_{[y]}\subseteq E(y,b)$ and $g(P_{[y]}\cong P_{[y]}\subseteq E(x,y)$.
\end{enumerate}
By Lemma \ref{lem: Px-Px isomorphism implies x<z<y} above, (a) implies that $x<a\le y$. Since $(Y,d)$ lies in $\cF_\phi$, we must have $\phi(x)\le a\le y$. Similarly, (b) implies $y< b\le x+2\pi$. Since $(Y,d)$ lies in $\cF_\phi$, this implies $\phi(y)\le b \le x+2\pi$. By Proposition \ref{prop: indecomposable objects of F-Phi} this implies that $E(x,y)$ lies in $\cF_\phi$. A similar argument shows that any kernel of a quotient map in $\cF_\phi$ lies in $\cF_\phi$.

To show that $E(x,y)$ is projective with respect to exact sequences in $\cF_\phi$ for $y=\phi(x)$, suppose that $p:(Y,d_Y)\onto (Z,d_Z)$ is a quotient map in $\cF_\phi$. Since $[x,y]$ does not properly contain any support interval for either $d_Y$ or $d_Z$, we have by Proposition \ref{prop: projective and injective property of E(x,y) in ccc-c} that $\cF_\pi(E(x,y),(Y,d_Y))=\cP_{S^1}(P_{[x]},Y)$ and $\cF_\pi(E(x,y),(Z,d_Z))=\cP_{S^1}(P_{[x]},Z)$. Since $Y\to Z$ is split epi, any morphism $E(x,y)\to (Z,d_Z)$ lifts to $(Y,d_Y)$. The dual argument using the second part of Proposition \ref{prop: projective and injective property of E(x,y) in ccc-c} proves that $E(x,y)$ is injective.

For any other indecomposable object $E(x,y)$ of $\cF_\phi$ we have, by Proposition \ref{prop: indecomposable objects of F-Phi}, that $\phi(x)\le y$ and $\phi(y)\le x+2\pi$. So, we have quotient map and cofibration:
$$E(x,\phi(x))\oplus E(y,\phi(y))\onto E(x,y), \ \ \ \ \ \ E(x,y)\cof E(\phi^{-1}(y),y) \oplus E(\phi^{-1}(x),x).$$
If $E(x,y)$ is projective or injective then either the first or second map is split and we get that $y=\phi(x)$. This show that we have enough projectives and enough injectives and that they all have the form $E(x,\phi(x))\cong E(\phi(x),x+2\pi)$.
\end{proof}

\begin{corollary}
The category $\cF_c$ is a Frobenius category with projective-injective objects $E(x,x+\th)\cong E(x+\th,x+2\pi)$.\qed
\end{corollary}


\section{Continuous cluster categories}\label{sec3:CCC} Categories $\cC_\pi$ and $\cC_c$ are defined here; we show that they are equivalent to the stable categories of the Frobenius categories $\cF_\pi$ and $\cF_c$, and therefore are triangulated by Happel's theorem. All the structure maps, including the triangulation maps, are continuous. In a subsequent paper \cite{IT10} we show that the category $\cC_\pi$ has cluster structure, hence the name: continuous cluster category.

A \emph{cluster structure} on a triangulated category is defined \cite{BIRSc} to be a collection of subsets called \emph{clusters} of the set of indecomposable objects satisfying four conditions. The first two are:
\begin{enumerate}
\item[(a)] For any element $T$ of any cluster $\cT$ there is, up to isomorphism, a unique object $T^\ast$ so that $\cT\backslash T\cup T^\ast$ is a cluster.
\item[(b)] There are distinguished triangles $T^\ast\to B\to T\to \Sig T^\ast$ and $T\to B'\to T^\ast\to \Sig T$ where $B$ is a right $add\,\cT\backslash T$ approximation of $T$ in the sense $B$ is an object of $add\,\cT\backslash T$, any morphism from an object of $\cT\backslash T$ into $T$ factors through $B$ and $B$ is minimal with this property and $B'$ is a left $add\,\cT\backslash T$ approximation of $T$.
\end{enumerate}
See (\cite{BIRSc}, section 1.1) for more details, including the other two conditions in the definition of a cluster.

\subsection{The stable category $\underline\cF_\pi$ and continuous cluster category $\cC_\pi$} We first recall some basic properties of the stable category $\underline\cF_\pi$, then define the continuous cluster category $\cC_\pi$. It will follow from the definition that $\underline\cF_\pi$ and $\cC_\pi$ are algebraically equivalent. We require $\cC_\pi$ to be a topological category with continuous triangulation. For this purpose we construct a topological category $\cF_\pi^{top}$ which is algebraically equivalent to $\cF_\pi$. The topologies on the stable category $\underline\cF_\pi^{top}$ and $\cC_\pi$ will be defined to be the quotient topologies induced by functors $\cF_\pi^{top}\onto \underline\cF_\pi^{top}\onto\cC_\pi$ which are epimorphisms on object and morphism sets.



Recall that the stable category $\underline\cF_\pi$ has the same objects as $\cF_\pi$ and the morphism sets $\underline\cF_\pi(X,Y)$ are quotients of $\cF_\pi(X,Y)$ modulo those morphisms which factor through projective-injective objects, which are direct sums of objects of the form:
 $$E(x,x)=\left(P_{[x]}^2, d=\small\mat{0&t\\1&0}\right).$$ For example, multiplication by $t$ is 0 in the stable category  $\underline\cF_\pi$ since $t=\small\mat{t&0\\0&t}$ factors as:
\[\xymatrixcolsep{9pt}
\xymatrix{
	E(x,y)\ar@{=}[r]&P_{[x]}\oplus P_{[y]}\ar[rrr]^{\small\mat{0 & \b_\ast\\1 & 0}}&&&P_{[x]}\oplus P_{[x]}\ar[rrr]^{\small\mat{0 & t\\\a_\ast & 0}}&&&P_{[x]}\oplus P_{[y]}\ar@{=}[r]&E(x,y)}.
\]
We also recall that $E(x,y)$ is isomorphic to $E(y,x+2\pi)$.

The following Lemmas will be used to prove the isomorphism on the Hom sets.

\begin{lemma}\label{second lemma for stable continuous Frob}
If $x<y\le a<b<x+2\pi$ then $\underline\cF_\pi(E(x,y),E(a,b))=0$.
\end{lemma}

\begin{proof}
All morphisms from $P_{[x]}$ to $ P_{[a]}\oplus P_{[b]}$ factor through $\a_\ast:P_{[x]}\to P_{[y]}$. So, any morphism $E(x,y)\to E(a,b)$ factors though $\a_\ast\oplus 1:E(x,y)\to E(y,y)$.
\end{proof}

\begin{lemma}\label{third lemma for stable continuous Frob}
(a) If $x\le a<y\le b<x+2\pi$ then $\underline\cF_\pi(E(x,y),E(a,b))=K$ is generated by $f\oplus g$ where $f:P_{[x]}\to P_{[a]}$ sends $e_x$ to $e_a^{a-x}$ and $g:P_{[y]}\to P_{[b]}$ sends $e_y$ to $e_{b}^{b-y}$ where $V=E(a,b)$.

(b) Furthermore, any nonzero multiple of $f\oplus g$ factors through $E(c,d)$ if and only if either $x\le c\le a<y\le d\le b<x+2\pi$ (for some choice of liftings of $c,d$ to $\RR$) or $x\le d\le a<y\le c+2\pi\le b<x+2\pi$ (the same condition with $(c,d)$ replaced by $(d,c+2\pi)$).
\end{lemma}

\begin{proof}
Any morphism $E(x,y)\to E(a,b)$ is a sum of two morphisms: a \emph{diagonal} and \emph{counter-diagonal} morphism:
\[
	\mat{f & h\\k& g}=\mat{f & 0\\0& g}+\mat{0 & h\\k& 0}
\]
The counter-diagonal morphism is stably trivial is this case since it factors through $E(y,y)$. The morphisms $f,g$ make the following diagram commute.
\[
\xymatrixrowsep{20pt}\xymatrixcolsep{20pt}
\xymatrix{
	P_{[x]}\ar[d]_{\a_\ast}\ar[r]^f&P_{[a]}\ar[d]^{\g_\ast}\\
	P_{[y]}\ar[r]^g & P_{[b]}
}
\]
So, if $f(e_x)=re_{a}^{a-x}$ and $g(e_y)=se_{b}^{b-y}$, then $\g_\ast(f(e_x))=re_{b}^{b-x}=g(\a_\ast(e_x))=se_{b}^{b-x}$ making $r=s$. If $r=s\in(t)$ then the morphism $f\oplus g$ is divisible by $t$ and is thus stably trivial. If $r=s\notin(t)$ then neither $f$ nor $g$ is divisible by $t$. In this case, $f:P_{[x]}\to P_{[a]}$ can only factor through $P_{[c]}$ where $x\le c\le a$ and $g:P_{[y]}\to P_{[b]}$ can only factor through $P_{[d]}$ where $y\le d\le b$. Thus $f\oplus g$ factors through $E(c,d)$ via diagonal morphisms if and only if $x\le c\le a<y\le d\le b<x+2\pi$ (for some choice of liftings of $c,d$ to $\RR$). $f\oplus g$ factors through $E(c,d)$ with counter-diagonal morphisms iff the other condition holds. Since the intervals $[x,a], [y,b]$ are disjoint on the circle $S^1$, $f\oplus g$ cannot factor through a projective-injective $P_{[z]}^2$. Therefore, $f\oplus g$ is not stably trivial. Since $r=s$, the stable hom set $\underline\cF_\pi(E(x,y),E(a,b))$ is one dimensional and generated by the morphism $f\oplus g$ with $r=s=1$ as claimed.
\end{proof}


The standard definition of $Ind\,\cF_\pi$ would choose one object from each isomorphism of indecomposable objects of $\cF_\pi$ and the space of such objects has a natural topology as a compact Moebius band (with boundary points being the projective-injective objects). However, we require \emph{two objects} in each isomorphism class of indecomposable objects for the following reason.

\begin{proposition}\label{two fold cover} Let $\underline{\cF}$ denote the stable category of the Frobenius category $\cF=add\,Ind\,\cF_\pi$. If $char(K)\neq 2$, it is not possible to put a continuous triangulated structure on $\underline\cF$ in such a way that the subcategory of indecomposable objects has the natural topology as an open Moebius band.
\end{proposition}

\begin{proof} This follows from the fact that $X\cong \Sig X$ for all objects $X$ and the fact that, for every nonzero morphism $f:X\to Y$ between indecomposable objects $X,Y$ there is a continuous path $f_s:X\to Y_s$, $0\le s\le 1$ of nonzero morphisms starting with $Y_1=Y,f_1=f$ and ending with $Y_0=X$. 

In $Ind\,\cF_\pi$, there is only one object in each isomorphism class. Therefore $\Sig X=X$. Being a functor, $\Sig $ must send $id_X$ to $id_X$. Being $K$-linear, this implies that $\Sig f=f$ for all endomorphisms $f$ of $X$. For any nonzero morphism $f:X\to Y$, choose a continuous path $f_s:X\to Y_s$ from $f_1=f$ to $f_0:X\to Y_0=X$. The operator $\Sig $ acts as a linear automorphism of $\underline\cF(X,Y_s)\cong K$. When $s=0$, $\Sig $ is the identity operator. By continuity of $\Sig $, it must be the identity operator on $\underline\cF(X,Y_s)$ for all $0\le s\le 1$. Therefore, $\Sig $ is the identity functor. This contradicts the axioms of a triangulated category when $char(K)\neq2$.
\end{proof}

This means  we need a 2-fold covering space of the Moebius band. Up to isomorphism there are two choices. For purely esthetic reasons, we choose the connected covering given by the following definition. (Remark \ref{another two fold cover} gives the other choice.)

\begin{definition}\label{defn: tilde Ind Ftop}
Let $\widetilde{Ind}\,\cF_\pi^{top}$ be the topological $R$-category whose object set is the set of all $E(x,y)$ with $x\le y\le x+2\pi$. This has two objects in every isomorphism class of indecomposable objects in $\cF_\pi$ since $E(x,y)\cong E(y,x+2\pi)$. We give this set the topology as a quotient space of a subspace of the plane:
\[
	Ob(\widetilde{Ind}\,\cF_\pi^{top})=\{(x,y)\in \RR^n\ |\ x\le y\le x+2\pi\}/\sim
\]
where the equivalence relation is $(x,y)\sim (x+2\pi n,y+2\pi n)$ for all $n\in \ZZ$. This is a compact Hausdorff space homeomorphic to the Cartesian product $S^1\times [0,2\pi]$.

The space of morphisms $Mor(\widetilde{Ind}\,\cF_\pi^{top})$ is not important but we specify it for completeness. It is the quotient space of the space of all 6-tuples $(r,s,(x_1,x_2),(y_1,y_2))\in R^2\times\RR^4$ satisfying the following closed conditions:
(1) $x_1\le x_2\le x_1+2\pi$,
(2) $y_1\le y_2\le y_1+2\pi$,
(3) Either $r=0$ or $y_1\ge x_1$ and $y_2\ge x_2$,
(4) Either $s=0$ or $y_2\ge x_1$ and $y_1+2\pi\ge x_2$.
The equivalence relation is given by $(r,s,(x_1,x_2),(y_1,y_2))\sim(r,s,(x_1+2\pi n,x_2+2\pi n),(y_1+2\pi n,y_2+2\pi n))$ for all $n\in\ZZ$ and $(r,s,(x_1,x_2),(y_1+2\pi,y_2+2\pi))\sim (tr,ts,(x_1,x_2),(y_1,y_2))$. The morphism $(r,s,X,Y)$ represents $r$ times the basic diagonal morphism $X\to Y$ plus $s$ times the basic counter-diagonal morphism. (Morphisms $X\to Y$ are given by $2\times 2$ matrices.)
\end{definition}

\begin{definition}
We define the topological Frobenius category 
$
 	\cF_\pi^{top}$ to be $add\,\widetilde{Ind}\,\cF_\pi^{top}
$
with topology given by Definition \ref{topology of add D}. The stable category $\underline\cF_\pi^{top}$ is given the \emph{quotient topology} defined as follows. The object space of $\underline\cF_\pi^{top}$ is equal to the object space of $\cF_\pi$ with the same topology. The morphism space $Mor(\underline\cF_\pi^{top})$ is the quotient space of $Mor(\cF_\pi^{top})$ modulo the standard equivalence relation reviewed above with the quotient topology.  Define the topological category $\underline{\widetilde{Ind}\,\cF}_\pi^{top}$ to be the image of ${\widetilde{Ind}\,\cF}_\pi^{top}$ in $\underline\cF_\pi^{top}$.
\end{definition}

\begin{proposition} The inclusion functor $\cF_\pi^{top}\to \cF_\pi$ is an equivalence of Frobenius categories. As topological categories, $add\,\underline{\widetilde{Ind}\,\cF}_\pi^{top}$ and $\underline\cF_\pi^{top}$ are isomorphic.
\end{proposition}

\begin{proof}
The first statement follows from the fact that both categories are Krull-Schmidt and every indecomposable object of $\cF_\pi$ is isomorphic to some object in $\widetilde{Ind}\cF_\pi^{top}$. The continuous isomorphism $F:add\,\underline{\widetilde{Ind}\,\cF}_\pi^{top}\to \underline\cF_\pi^{top}$ is induced by the continuous inclusion functor $\underline{\widetilde{Ind}\,\cF}_\pi^{top}\into \underline\cF_\pi^{top}$. The functor $F$ is an isomorphism since $\underline\cF_\pi^{top}$ has the same space of objects as $add\,\widetilde{Ind}\, \cF_\pi^{top}$. Using the fact that any continuous bijection from a compact space to a Hausdorff space is a homeomorphism, we conclude that $F$ is a topological isomorphism of categories.
\end{proof}

By definition the topological stable category $\underline\cF_\pi^{top}$ has an infinite number of (isomorphic) zero objects. When we identify all of these objects to one point and take the quotient topology, we get the continuous cluster category. 

\begin{definition}\label{defn: continuous cluster category} For any field $K$ the \emph{continuous cluster category} $\cC_\pi$ is the additive category generated by the $K$-category $\widetilde{Ind}\,\cC_\pi$ defined as follows.

The \emph{object set} of $\widetilde{Ind}\,\cC_\pi$ will be the set of all ordered pairs of distinct points in $S^1$. Objects are labeled by pairs of real numbers $(x,y)$ with $x<y<x+2\pi$ with $(x,y)=(x+2\pi n,y+2\pi n)$ for all integers $n$.

\emph{Morphism set}: For any two objects $X,Y$ in $\widetilde{Ind}\,\cC_\pi$, we define $\cC_\pi(X,Y)$ to be a one-dimensional vector space if the coordinates $X=(x_0,x_1), Y=(y_0,y_1)$ can be chosen to satisfy either $x_0\le y_0<x_1\le y_1<x_0+2\pi$ or $x_0\le y_1<x_1\le y_0+2\pi<x_0+2\pi$. We denote the generator of $\cC_\pi(X,Y)$ by $b^{XY}_+$ in the first case and $b^{XY}_-$ in the second case. Nonzero morphisms are $rb^{XY}_\e$ where $r\neq0\in K$ and $\e=\pm$. The composition of morphisms $rb^{XY}_\e:X\to Y$ and $sb^{YZ}_{\e'}:Y\to Z$ is defined to be $rsb^{XZ}_{\e\e'}:X\to Z$ provided that $\cC_\pi(X,Z)$ is nonzero with generator $b^{XZ}_{\e\e'}$. Otherwise the composition is zero. We call the chosen morphisms $b^{XY}_\e$ \emph{basic morphisms}. Any composition of basic morphisms is either basic or zero.
\end{definition}

\begin{center}
\begin{figure}[ht]\label{fig0}
\centering
\subfigure{
       \includegraphics[width=2in]{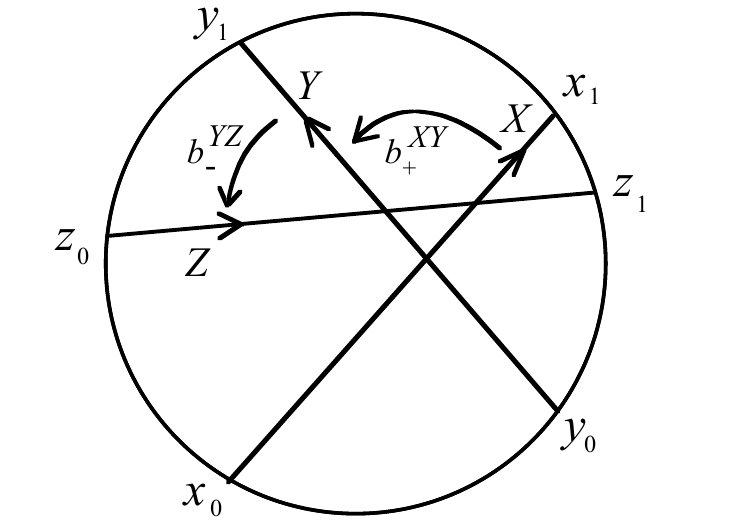}
}
\subfigure{
\setlength{\unitlength}{1in}
{\mbox{
\begin{picture}(2,1.3)
\put(0,.2){
	\put(0,.6){$ b_+^{XY}:X=(x_0,x_1)\to Y=(y_0,y_1)$}
	\put(0,.2){$ b_-^{YZ}:Y=(y_0,y_1)\to Z=(z_0,z_1)$}
	}
\end{picture}}}
} 
\caption{In this figure, $b_+^{XY}$ has subscript $(+)$ since $x_0\le y_0<x_1\le y_1<x_0+2\pi$. This is equivalent to the condition that when $X$ is rotated counterclockwise its orientation matches that of $Y$. The morphism $b_-^{YZ}$ has negative subscript since $y_0\le z_1<y_1\le z_0<y_0+2\pi$ and this is equivalent to saying that the orientation of $Y$ when rotated counterclockwise does not match that of $Z$.
}
\end{figure}
\end{center}

\begin{remark}
The subscript $\e$ in $b_\e^{XY}:X\to Y$ is uniquely determined by the objects $X,Y$. Its purpose is to give the formula for the shift functor in the triangulated structure of the category. We will see later (Remark \ref{rem T(be)=eb}) that $\Sig X=(y,x)$ if $X=(x,y)$ and $\Sig (rb_\e^{XY}):\Sig X\to \Sig Y$ is equal to $\e rb_\e^{\Sig X,\Sig Y}$. \end{remark}


\begin{theorem}\label{thm: stable cat of continuous Frob is cont cluster}
The stable category of the continuous Frobenius category $\cF_\pi$ is equivalent, as a $K$-category, to the continuous cluster category $\cC_\pi$, i.e.
$\underline\cF_\pi\approx \cC_\pi$.
\end{theorem}
\begin{proof}
Since $\underline\cF_\pi\cong \underline\cF_\pi^{top}=add\,\underline{\widetilde{Ind}\,\cF}_\pi^{top}$ and $\cC_\pi=add\,\widetilde{Ind}\,\cC_{\pi}$, it suffices to show that $\underline{\widetilde{Ind}\,\cF}_\pi^{top}$ is equivalent to $\widetilde{Ind}\,\cC_{\pi}$ as $K$-categories. We will show that the full subcategory $\underline{\widetilde{Ind}_\ast\cF}_\pi$ of nonzero objects of $\underline{\widetilde{Ind}\,\cF}_\pi^{top}$, i.e., those $E(x,y)$ with $0\le x<2\pi$ and $x< y<x+2\pi$, with no topology, is isomorphic to $\widetilde{Ind}\,\cC_\pi$. 
An isomorphism 
\[
	\Psi:\underline{\widetilde{Ind}_\ast\cF}_\pi\xrarrow\cong \widetilde{Ind}\,\cC_\pi
\] 
is given on objects by $\Psi E(x,y)=(x,y)\in S^1\times S^1$ considered as an object of $\widetilde{Ind}\,\cC_\pi$. This is a bijection on objects.
By Lemma \ref{second lemma for stable continuous Frob} and Lemma \ref{third lemma for stable continuous Frob}(a), $\Psi$ is an isomorphism on Hom sets. By Lemma \ref{third lemma for stable continuous Frob}(b), the composition of two basic morphisms $E(x,y)\to E(c,d)\to E(a,b)$ is a nonzero basic morphism if and only if $x\le c\le a<y\le d\le b<x+2\pi$ (for some choice of liftings to $\RR$) or the analogous condition holds with $(c,d)$ replaced by $(d,c+2\pi)$ or with $(a,b)$ replaced by $(b,a+2\pi)$. So, $\Psi$ respects composition. It is an isomorphism of $K$-categories. 
\end{proof}

\begin{remark}\label{topology of Cpi}
The functor $\Psi$ extends to an equivalence of categories $\underline\cF_\pi^{top}\to \cC_\pi$ which is surjective on object and morphism sets. We give $\cC_\pi$ the quotient topology. Readers familiar with topology will recognize that this is the James construction making the object set of $\cC_\pi$ homotopy equivalent to the connected space $\Omega\Sigma(S^2\vee S^1)$. (See \cite{Hatcher}, p. 224.) However, the object set of $\underline\cF_\pi^{top}$ is disconnected since, by definition, it is the same as the object set of $\cF_\pi^{top}=add\,\widetilde{Ind}\cF_\pi^{top}$ which is defined as a disjoint union (Definition \ref{topology of add D}).
\end{remark}


\subsection{{Stable categories $\underline\cF_c$ and $\underline\cF_\phi$ and the continuous categories $\cC_c$ and $\cC_\phi$}}\ 

\noindent The Frobenius categories $\cF_c$ and $\cF_\phi$ were defined and studied in the previous section. Here we define  categories $\cC_c$ and $\cC_\phi$ and show that they are equivalent to the stable categories of $\cF_c$ and $\cF_\phi$. The  categories $\cC_c$ and $\cC_\phi$ are continuous triangulated categories, however they are not necessarily cluster categories.  

In a subsequent paper \cite{IT10} we show that $\cC_c$ has cluster structure precisely when $c=(n+1)\pi/(n+3)$ for $n\in\NN$. In that case every cluster is finite and each cluster is contained in a thick subcategory of $\cC_c$ which is equivalent to the cluster category of type $A_n$. The $\cC_\phi$ construction is more versatile and has a cluster structure whenever $\phi$ has fixed points. For example, if $\phi$ has exactly one fixed point then $\cC_\phi$ contains the cluster category of type $A_\infty$ as a thick subcategory. Furthermore, $\cC_\phi$ contains an $m$-cluster category of type $A_\infty$ as a thick subcategory if $\phi$ has exactly $m$ points of finite period, say $x_i=\phi^i(x_0)$ with $x_m=x_0$, $m\ge3$ and so that $x_0,y,\phi^m(y),x_1$ are in cyclic order for all $y$ between $x_0$ and $x_1$. These clusters and $m$-clusters of type $A_\infty$ will have an infinite number of objects and will be explored in \cite{IT11}.

Since $\cC_c$ are special cases of $\cC_\phi$, we give the definition only in the second case.

\begin{definition} Let $K$ be a field and $\phi$ a $2\pi$-periodic homeomorphism of $\RR$ as in Definition \ref{def: Phi and F-Phi}. The \emph{continuous category} $\cC_\phi$ is defined to be the additive category generated by $\widetilde{Ind}\,\cC_\phi$ where $\widetilde{Ind}\,\cC_\phi$ denotes the category with:\\
1.  Objects are ordered pairs of points $(x_0,x_1)$ in $S^1$ so that $x_0\le \phi(x_0)\le x_1\le \phi^{-1}(x_0+2\pi)$.
\\
2. Morphisms are given by
\[
	\cC_\phi(X,Y)=\begin{cases} Kb^{XY}_+ & \text{if } x_0\le y_0<\phi^{-1}(x_1)\le x_1\le y_1<\phi^{-1}(x_0+2\pi)\\
	Kb^{XY}_- & \text{if } x_0\le y_1<\phi^{-1}(x_1)\le x_1\le y_0+2\pi<\phi^{-1}(x_0+2\pi)\\
  0  & \text{if the elements of $X,Y$ do not lift to such real numbers}
    \end{cases}
\]
3. Composition of morphisms is given by 
\[
	rb^{YZ}_{\e}\circ sb^{XY}_{\e'}=\begin{cases} rsb^{XZ}_{\e\e'} & \text{if } \cC_c(X,Z)=Kb^{XZ}_{\e\e'}\\
  0  & \text{otherwise}
    \end{cases}
\]
For $0<c\le\pi$ the category $\cC_c$ is defined to be $\cC_\phi$ in the case when $\phi(x)=x+\pi-c$ for all $x\in\RR$. The topology on these categories is given in Definition \ref{def: topology of Cc,Cphi}.
\end{definition}

\begin{proposition}
The stable category of the Frobenius category $\cF_\phi$ is equivalent to the category $\cC_\phi$.
\end{proposition}

\begin{proof}
The verification of the proposition follows the pattern of Theorem \ref{thm: stable cat of continuous Frob is cont cluster} and is straightforward.
\end{proof}

\begin{corollary}
For any positive $c<\pi$, the stable category of the Frobenius category $\cF_c$ is equivalent to the category $\cC_c$.
\end{corollary}

\begin{definition}\label{def: topology of Cc,Cphi}
The category $\cF_\phi$ is a full subcategory of $\cF_\pi$ and we define $\cF_\phi^{top}$ to be the corresponding topological full subcategory of $\cF_\pi^{top}$. We have functors $\cF_\phi^{top}\to \underline\cF_\phi^{top}\to \cC_\phi$ which are epimorphisms on objects and morphisms. So, we define $\underline\cF_\phi$, $\cC_\phi$ to have the quotient topologies with respect to $\cF_\phi^{top}$. As a special case we obtain the topological categories $\cF_c^{top}$, $\underline\cF_c^{top}$ and $\cC_c$.
\end{definition}

\subsection{Distinguished triangles}\label{ss:distinguished triangles} In order to obtain an explicit triangulation of the stable category of any Frobenius category we need to fix a choice, for each object $X$, of an exact sequence $X\to P\to Y$ where $P$ is projective-injective. In the Frobenius category $\cF_\pi$, for each indecomposable nonprojective object $E(x,y)$ in $\cF_\pi$ we choose the following exact sequence.
\begin{equation}\label{injective envelope sequence of E(x,y)}
	E(x,y) \xrarrow{\small\left[\begin{matrix}{1}\\{-1}\end{matrix}\right]}E(y,y)\oplus E(x,x+2\pi)\xrarrow{[1,1]}E(y,x+2\pi)
\end{equation}
Here all morphisms between indecomposable objects $E(a,b)$ are diagonal (as $2\times 2$ matrices) and labeled by the scalar $r\in R$ indicating that they are $r$ times the basic diagonal morphism.
The middle term is projective-injective. So, this choice defines the shift functor
\[
	\Sig E(x,y)=E(y,x+2\pi)
\]
in the stable category $\underline\cF_\pi$ of $\cF_\pi$. (See \cite{HappelBook} for details.) The functor $\Sig =[1]$ takes basic diagonal morphisms to basic diagonal morphisms and takes basic counter-diagonal morphisms to negative basic counter-diagonal morphisms. For example, a basic counter-diagonal morphism $f:E(x,y)\to E(a,b)$ takes $P_{[x]}$ to $P_{[b]}$ and $P_{[y]}$ to $P_{[a]}$ and therefore induces a map of diagrams:
\[
\xymatrixrowsep{35pt}\xymatrixcolsep{45pt}
\xymatrix{
E(x,y)\ar[d]^1\ar[r]^(.35){\small\left[\begin{matrix}{1}\\{-1}\end{matrix}\right]} &
	E(y,y)\oplus E(x,x+2\pi)\ar[d]^{\small\mat{0&-1\\-1&0}}\ar[r]^(.6){[1,1]} &
	E(y,x+2\pi)\ar[d]^{-1}\\
E(a,b) \ar[r]^(.35){\small\left[\begin{matrix}{1}\\{-1}\end{matrix}\right]}& 
	E(b,b) \oplus E(a,a+2\pi)\ar[r]^(.6){[1,1]}&
	E(b,a+2\pi)
	}
\]
where the horizontal morphisms are all basic diagonal morphisms times the indicated scalars and the vertical morphisms are basic counter-diagonal morphisms times the indicated scalars. For example, the morphism $-1:E(y,y)\to E(a,a+2\pi)$ is the composition of the basic counter-diagonal isomorphism $E(y,y)\to E(y,y+2\pi)$ composed with $-1$ times the basic diagonal morphism $E(y,y+2\pi)\to E(a,a+2\pi)$. This diagram shows that the functor $\Sig $ takes the basic counter-diagonal morphism $E(x,y)\to E(a,b)$ to $-1$ times the basic counter-diagonal morphism $\Sig E(x,y)\to \Sig E(a,b)$.

\begin{remark}\label{rem T(be)=eb}
In terms of the abstractly defined category $\cC_\pi$ of Definition \ref{defn: continuous cluster category}, the functor $\Sig $ acts on objects by $\Sig (x,y)=(y,x)$ and on basic morphisms by $\Sig (b_\e^{XY})=\e b_\e^{\Sig X,\Sig Y}$ since $\e=+$ for diagonal morphisms and $\e=-$ for counter-diagonal morphisms.
\end{remark}

As an example of the general construction and the detailed knowledge we obtain about the triangulated structure of the continuous cluster category, we give the following example of a distinguished triangle in $\cC_\pi$.

\begin{example}\label{eg: distinguished triangle} Take any 6 distinct points $a<b<c<x<y<z<a+2\pi$. 
Let $X=(a,x)$, $Y=(b,z)\oplus(c,y)$. Then any morphism $\underline f:X\to Y$, both of whose components are nonzero, can be completed to a distinguished triangle
\begin{equation}\label{eq:distinguished triangle example}
	X\xrarrow{\underline f}Y\xrarrow{\underline g}Z\xrarrow{\underline h}\Sig X
\end{equation}
if and only if $Z\cong (b,a+2\pi)\oplus (c,z)\oplus (x,y)$. Furthermore, this triangle will be distinguished if and only if the matrices $(f_i),(g_{ji}),(h_j)$ of the morphisms $\underline f,\underline g,\underline h$ satisfy the six conditions listed in the Figure 2 as it is proved below.
\begin{center}
\begin{figure}[ht]\label{fig22}
\centering
\subfigure{
       \includegraphics[width=2in]{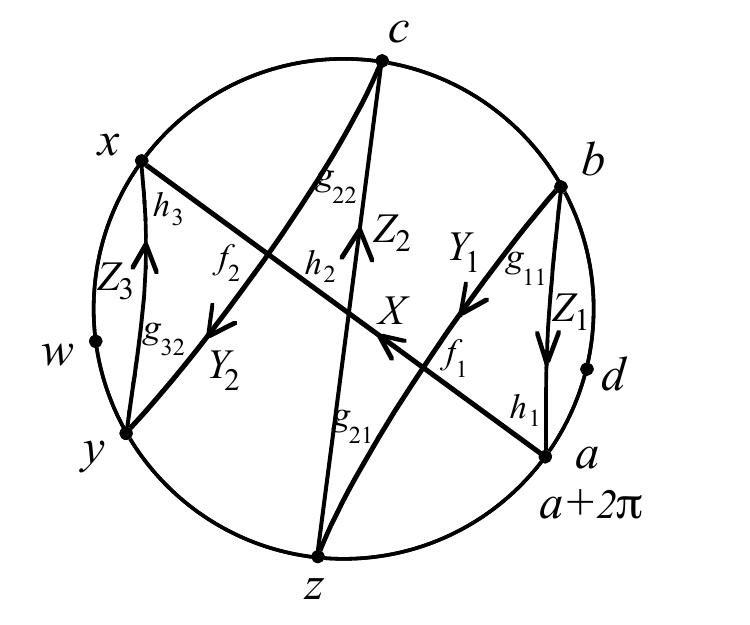}
}
\subfigure{
\setlength{\unitlength}{1in}
{\mbox{
\begin{picture}(1.7,1.3)
\put(.3,.5){
	\put(0,.8){1) \ $ h_1g_{11}f_1=-1$}
	\put(0,.6){2) \ $ h_2g_{21}f_1=1$}
	\put(0,.4){3) \ $ h_2g_{22}f_2=-1$}
	\put(0,.2){4) \ $ h_3g_{32}f_2=1$}
	\put(0,.0){5) \ $ g_{12}=0$}
	\put(0,-.2){6) \ $ g_{31}=0$}
	}
\end{picture}}}
} 
\caption{In the four triangles in the figure, the product of the angles is 1 or -1 depending on the orientation of $X$ around its boundary. $g_{12}=0$ since $Y_2$ does not meet $Z_1$. Similarly, $g_{31}=0$.}
\end{figure}
\end{center}
This is an example of a general procedure for determining which candidate triangles are distinguished. For example, the following statements are proved in \cite{IT1y}:
\begin{enumerate}
\item Given any distinguished triangle in $\cC_\pi$, vertices can be merged together and the result will still be a distinguished triangle. For example $z,a,b$ can be merged to one point, say $a$, and $x,y$ can be merged to one point, say $x$, giving the distinguished triangle:
\begin{equation}\label{dist triangle coming from collapsing example}
	X=(a,x)\xrarrow{f_2} Y_2=(c,x)\xrarrow{g_{22}}Z_2=(a,c)\xrarrow{h_2}\Sig X=(x,a)
\end{equation}
\item Three morphisms $(a,b)\to (b,c)\to (c,a)\to \Sig (a,b)=(b,a)$ with $a,b,c$ in cyclic order on the circle form a distinguished triangle if and only if the product of the scalars corresponding to the three maps is $+1$. If $a,b,c$ are not in cyclic order, the three morphisms form a distinguished triangle if and only if the product of the scalars corresponding to the maps is $-1$. For example, in \eqref{dist triangle coming from collapsing example} we must have $h_2g_{22}f_2=-1$ since $a,x,c$ are not in cyclic order. (The orientation of the two middle terms $Y_2,Z_2$ is irrelevant.)
\end{enumerate}

This example also gives an example of a distinguished triangle in the cluster category $\cC_{A_5}$ of type $A_5$. For this interpretation we need to add the two points $d,w$ in the centers of arcs $Z_1,Z_3$ to make these objects nonzero. Let $\phi$ be an automorphism of the circle sending $a,d,b,c,x,w,y,z,a$ to the next point in the sequence. Then the thick subcategory of $\cC_\phi$ consisting of objects with endpoints in this set of 8 points is equivalent to $\cC_{A_5}$. In this triangulated category, the definition of $\Sig$ is not the same as in $\cC_\pi$. So, we need to replace $X$ with $\tau X=W=(d,w)$. Then $\Sig W=(x,a)$ in $\cC_{A_5}$, and the distinguished triangle is:
\begin{equation}\label{dist triangle in A5}
	W=(d,w)\xrarrow{\underline f'} Y\xrarrow{\underline g} Z\xrarrow{\underline h} \Sig W=(x,a)
\end{equation}
Here $g,h$ are the same as in \eqref{eq:distinguished triangle example}. But $f'$ is the composition of $f:X\to Y$ with the basic map $W\to X$. The coordinates of the matrices of $\underline{f}',\underline g,\underline h$ are the same as for $\underline f,\underline g,\underline h$ and \eqref{dist triangle in A5} is a distinguished triangle if and only if these coefficients satisfy the six conditions in Figure 2.
\end{example}

\begin{proof} By definition the distinguished triangles starting with $\underline f:X\to Y$ is given by lifting $\underline f$ to a morphism in $\cF_\pi$ and taking the pushout of the chosen sequence for $X$:
\[
\xymatrix{
 X=E(a,x)\ar[d]_f\ar[r]^(.4){\small\left[\begin{matrix}{1}\\{-1}\end{matrix}\right]} &
	E(x,x)\oplus E(a,a+2\pi)\ar[d]_u\ar[r]^{[1,1]} &
	\Sig X=E(x,a+2\pi)\ar[d]_=
	\\
 Y=E(b,z)\oplus E(c,y) \ar[r]^(.6)g& 
	Z \ar[r]^h&
	E(x,a+2\pi)
	}
\]
Since the bottom row is an exact sequence in $\cF_\pi$, as an object of $\cP_{S^1}$ we must have $Z= P_{[b]}\oplus P_{[c]}\oplus P_{[x]}\oplus P_{[y]}\oplus P_{[z]}\oplus P_{[a]}$ and these objects must be paired to make 3 indecomposable objects of $\cF_\pi$. One of them must be $E(b,w)$ for some $w$. But the morphism $g$ maps the summand $P_{[b]}$ of $E(b,z)$ isomorphically onto the summand $P_{[b]}$ of $E(b,w)$ and this is only possible if $w=z$ or $w=a+2\pi$. The first case is not possible since $\cC_\pi((c,y),(b,z))=0$ which would force $f_1=0$ contrary to assumption. So, $w=a+2\pi$ and $E(b,a+2\pi)$ is a summand of $Z$. An analogous argument shows that $E(x,y)$ is also a summand of $Z$ and the remaining two points must be paired to give $E(c,z)$ as claimed.

Next we show that the 6 conditions are sufficient to have a distinguished triangle. To do this we first lift the elements $f_i,g_{ji},h_j\in K$ to $R$ so that the 6 conditions are still satisfied. Then we let $u$ be given by the $3\times 2$ matrix with entries in $R$ given by:
\[
	u=\left[
	\begin{matrix}
	0 & g_{11}f_1\\0&0\\g_{32}f_2&0
	\end{matrix}
	\right]
\] 
Then the diagram commutes and $Y\to Z\to \Sig X$ is the pushout of the chosen sequence for $X$ making $X\to Y\to Z\to \Sig X$ a distinguished triangle by definition.

Finally, we prove the necessity of the 6 listed conditions. The last condition follows from the fact that $\underline\cF_\pi(Y_2,Z_1)=0=\underline\cF_\pi(Y_1,Z_3)$. The condition $\underline g\underline f=0$ implies that $g_{21}f_1+g_{22}f_2=0$. So, $(2),(3)$ are equivalent. The condition $\underline h\underline g=0$ implies that $h_1g_{11}+h_2g_{21}=0$ and $h_2g_{22}+h_3g_{32}=0$. So, $(1)-(4)$ are all equivalent. So, it suffices to prove (1).

The composition $X=E(a,x)\xrarrow{f_1} Y_1=E(b,z)\xrarrow{g_{11}} Z_1=E(b,a+2\pi)$ factors through $E(a,a+2\pi)$. Modulo the maximal ideal, the induced morphism $E(a,a+2\pi)\to Z_1$ must be $-g_{11}f_1$ times the basic morphism since the chosen map $X=E(a,x)\to E(a,a+2\pi)$ in \eqref{injective envelope sequence of E(x,y)} is $-1$ times the basic map. However, the composition $E(a,a+2\pi)\to Z_1\xrarrow{h_1} \Sig X$ is equal to the basic map $E(a,a+2\pi)\to \Sig X$ since $Y\to Z\to \Sig X$ is the pushout of \eqref{injective envelope sequence of E(x,y)} by definition of distinguished triangles. Therefore, $h_1(-g_{11}f_1)=1$ proving (1). So, the six conditions are both necessary and sufficient for $X\to Y\to Z\to \Sig X$ to be a distinguished triangle.
\end{proof}

\subsection{Generalizations} The following Proposition is a generalization of our previous construction of Frobenius categories, and the proof is essentially the same as the proof of Theorem \ref{thmxxx}. However we are not able to prove Krull-Schmidt property in this generality. The reason we include this here is the fact that a special case of this more general construction yields a triangulated category which is algebraically triangulated-equivalent to the continuous cluster category $\mathcal C_{\pi}$, however the two categories are not topologically equivalent, which we explain in the Remark below.

\begin{proposition}\label{prop: general Frobenius category} Let $R$ be a discrete valuation ring with uniformizer $t$. Let $\cP$ be any additive Krull-Schmidt $R$-category where the endomorphism rings of indecomposable objects are commutative local $R$-algebras. Let $\cF$ be the category of all pairs $(V,d)$ where $V$ is an object of $\cP$ and $d$ is an endomorphism of $V$ with $d^2$ equal to multiplication by the uniformizer $t\in R$. Take exact sequences in $\cF$ to be sequences $(X,d)\to (Y,d)\to (Z,d)$ where $X\to Y\to Z$ is split exact in $\cP$. Then $\cF$ is a Frobenius category with projective-injective objects the direct summands of $\left(X^2,\small\left[\begin{matrix}0 & t\\1 & 0\end{matrix}\right]\right)$ for $X$ in $\cP$. \qed
\end{proposition}

With the following choice of two-way approximation for each object $(V,d)$
\begin{equation}\label{eq: Orlov's definition of triangulated structure}
	(V,d)\xrarrow{\small\left[\begin{matrix}{d}\\{1}\end{matrix}\right]} \left(V^2,\small\left[\begin{matrix}0 & t\\1 & 0\end{matrix}\right]\right)\xrarrow{[-1,d]} (V,-d)
\end{equation}
we get the shift functor $\Sig $ in the stable category $\cC=\underline{\cF}$, to satisfy $\Sig (V,d)=(V,-d)$.

\begin{remark}\label{another two fold cover} The above choice of the two-way approximations comes from \cite{Orlov}. It gives a different topology on the continuous cluster category $C_{\pi}$ defined as follows. Let $\cM$ be the space of all two-element subsets $X=\{x,y\}\subset S^1$. Each $X\in\cM$ corresponds to the indecomposable object $(V,d)$ where $V=P_{[x]}\oplus P_{[y]}$ and $d:V\to V$ is the direct sum of the basic maps $P_{[x]}\to P_{[y]}$ and $P_{[y]}\to P_{[x]}$. These objects form a subspace of the space of objects homeomorphic to $\cM$. The definition $\Sig (V,d)=(V,-d)$ implies that $-d:V\to V$ is the direct sum of $-1$ times the basic maps $P_{[x]}\to P_{[y]}$ and $P_{[y]}\to P_{[x]}$. Since we are taking the discrete topology on the ground field $K$, there is no continuous path from $+1$ to $-1$, if  $char(K)\neq 2$. Therefore, the set of objects $(V,-d)$ forms a disjoint copy of the Moebius band $\cM$.

The space of indecomposable objects for this model of the cluster category $\cC_\pi$ is therefore homeomorphic to the disjoint union of two copies of the open Moebius band, as opposed to the connected (and oriented) two-fold covering which we are using. The reason is that, instead of changing the sign of $d$, we take ordered sets and define the shift $\Sig =[1]$ by changing the order of the letters $x,y$. Since one can go continuously from the ordered set $(x,y)$ to the ordered set $(y,x)$ by rotation of the chord, $(V,d)$ and $\Sig (V,d)$ are in the same connected component of the space of objects in the model for the continuous cluster categories $\cC_c$ constructed and described in detail in this paper.
\end{remark}



\begin{thebibliography}{aa}

\bibitem{BIRSc}
A.~B. Buan, O.~Iyama, I.~Reiten, and J.~Scott, \emph{Cluster structures for 2-{C}alabi-{Y}au categories and unipotent groups}, Compos. Math. \textbf{145} (2009), no.~4, 1035--1079.


\bibitem{BMRRT}
Aslak~Bakke Buan, Robert Marsh, Markus Reineke, Idun Reiten, and Gordana Todorov, \emph{Tilting theory and cluster combinatorics}, Adv. Math. \textbf{204} (2006), no.~2, 572--618.

  
  \bibitem{CCS}
P.~Caldero, F.~Chapoton, and R.~Schiffler, \emph{Quivers with relations arising from clusters ({$A_n$} case)}, Trans. Amer. Math. Soc. \textbf{358} (2006), no.~3, 1347--1364.
  
\bibitem{HappelBook}
Dieter Happel, \emph{Triangulated categories in the representation theory of finite dimensional algebras}, London Math. Soc. Lecture Note Ser., vol. \textbf{119}, Cambridge Univ. Press, Cambridge, 1988.
 
 
 \bibitem{Hatcher} Allen Hatcher, \emph{Algebraic Topology}, Cambridge Univ. Press, Cambridge, 2002





 \bibitem{HJ12}
 T. Holm and P. J\o rgensen, \emph{On a cluster category of infinite Dynkin type, and the relation to triangulations of the infinity-gon}, Math. Z. \textbf{270} (2012), 277--295.
 
 \bibitem{IT10} K. Igusa and G. Todorov, \emph{Continuous cluster categories I}, arXiv:1209.1879.
 
\bibitem{IT1x} K. Igusa and G. Todorov, \emph{Continuous cluster categories II: Continuous cluster-tilted categories}, in preparation.

\bibitem{IT1y} K. Igusa and G. Todorov, \emph{Distinguished triangles in the continuous cluster category}, in preparation.

\bibitem{IT11} K. Igusa and G. Todorov, \emph{Cluster categories coming from cyclic posets}, in preparation.


 
\bibitem{K} B. Keller, \emph{Chain complexes and stable categories}, Manuscripta Math \textbf{67} (1990), 379--417.

\bibitem{KellerReiten} B. Keller and I. Reiten, \emph{Acyclic Calabi-Yau categories, with an appendix by Michel Van den Bergh}, Compos. Math. 144 (2008), 1332--1348.

\bibitem{Ng} P. Ng, \emph{A characterization of torsion theories in the cluster category of Dynkin type {$A_{\infty}$}}, arXiv:1005.4364v1.

\bibitem{Orlov}
D.~O. Orlov, \emph{Triangulated categories of singularities and {D}-branes in {L}andau-{G}inzburg models}, Ð (Russian) Tr. Mat. Inst. Steklova \textbf{246} (2004), 3, 240--262.; transl. Proc. Steklov Inst. Math. \textbf{246} (2004), 3, 227Ð248.
 

\bibitem{vR} Adam-Christiaan van Roosmalen, \emph{Hereditary uniserial categories with Serre duality}, arXiv:1011.6077v1.


\bibitem{W}
Friedhelm Waldhausen, \emph{Algebraic ${K}$-theory of spaces}, Algebraic and Geometric Topology (New Brunswick, N.J., 1983), Lecture Notes in Math \textbf{1126}, Springer, Berlin, 1985, pp.~318--419.
\end{thebibliography}
\end{document}